\documentclass[11pt,leqno]{article}
\usepackage{graphicx, amsfonts, amsthm, amsxtra, amssymb, verbatim, makeidx, relsize}
\usepackage{float}
\usepackage[mathscr]{euscript}
\usepackage{hyperref}
\makeindex
\makeglossary
\begin{document}
\newtheorem{theo}{Theorem}
\newtheorem{theoconj}{Theorem-Conjecture}
\newtheorem{exam}{Example}
\newtheorem{coro}{Corollary}
\newtheorem{defi}{Definition}
\newtheorem{prob}{Problem}
\newtheorem{lemm}{Lemma}
\newtheorem{prop}{Proposition}
\newtheorem{rem}{Remark}
\newtheorem{conj}{Conjecture}
\newtheorem{calc}{}

\def\gru{\mu} 
\def\pg{{ \sf S}}               
\def\TS{\mathlarger{\bf T}}                
\def\NB{{\mathlarger{\bf N}}}
\def\group{{\sf G}}
\def\NLL{{\rm NL}}   

\def\plc{{ Z_\infty}}    
\def\pola{{u}}      
\newcommand\licy[1]{{\mathbb P}^{#1}} 
\newcommand\aoc[1]{Z^{#1}}     
\def\HL{{\rm Ho}}     
\def\NLL{{\rm NL}}   

\def\Z{\mathbb{Z}}                   
\def\Q{\mathbb{Q}}                   
\def\C{\mathbb{C}}                   
\def\N{\mathbb{N}}                   
\def\uhp{{\mathbb H}}                
\def\A{\mathbb{A}}                   
\def\dR{{\rm dR}}                    
\def\F{{\cal F}}                     
\def\Sp{{\rm Sp}}                    
\def\Gm{\mathbb{G}_m}                 
\def\Ga{\mathbb{G}_a}                 
\def\Tr{{\rm Tr}}                      
\def\tr{{{\mathsf t}{\mathsf r}}}                 
\def\spec{{\rm Spec}}            
\def\ker{{\rm ker}}              
\def\GL{{\rm GL}}                
\def\ker{{\rm ker}}              
\def\coker{{\rm coker}}          
\def\im{{\rm Im}}               
\def\coim{{\rm Coim}}            
\def\p{{\sf  p}}
\def\U{{\cal U}}   

\def\weig{{\nu}}
\def\r{{ r}}                       
\def\k{{\sf k}}                     
\def\ring{{\sf R}}                   
\def\X{{\sf X}}                      
\def\Ua{{   L}}                      
\def\T{{\sf T}}                      
\def\asone{{\sf A}}                  

\def\Ts{{\sf S}}
\def\cmv{{\sf M}}                    
\def\BG{{\sf G}}                       
\def\podu{{\sf pd}}                   
\def\ped{{\sf U}}                    
\def\per{{\bf  P}}                   
\def\gm{{  A}}                    
\def\gma{{\sf  B}}                   
\def\ben{{\sf b}}                    

\def\Rav{{\mathfrak M }}                     
\def\Ram{{\mathfrak C}}                     
\def\Rap{{\mathfrak G}}                     

\def\mov{{\sf  m}}                    
\def\Yuk{{\sf C}}                     
\def\Ra{{\sf R}}                      
\def\hn{{ h}}                         
\def\cpe{{\sf C}}                     
\def\g{{\sf g}}                       
\def\t{{\sf t}}                       
\def\pedo{{\sf  \Pi}}                  

\def\Der{{\rm Der}}                   
\def\MMF{{\sf MF}}                    
\def\codim{{\rm codim}}                
\def\dim{{\rm    dim}}                
\def\Lie{{\rm Lie}}                   
\def\gg{{\mathfrak g}}                

\def\u{{\sf u}}                       

\def\imh{{  \Psi}}                 
\def\imc{{  \Phi }}                  
\def\stab{{\rm Stab }}               
\def\Vec{{\rm Vec}}                 
\def\prim{{\rm  0}}                  

\def\Fg{{\sf F}}     
\def\hol{{\rm hol}}  
\def\non{{\rm non}}  
\def\alg{{\rm alg}}  
\def\tra{{\rm tra}}  

\def\bcov{{\rm \O_\T}}       

\def\leaves{{\cal L}}        

\def\cat{{\cal A}}              
\def\im{{\rm Im}}               

\def\pn{{\sf p}}              
\def\Pic{{\rm Pic}}           
\def\free{{\rm free}}         
\def \NS{{\rm NS}}    
\def\tor{{\rm tor}}
\def\codmod{{\xi}}    

\def\GM{{\rm GM}}

\def\perr{{\sf q}}        
\def\perdo{{\cal K}}   
\def\sfl{{\mathrm F}} 
\def\sp{{\mathbb S}}  

\newcommand\diff[1]{\frac{d #1}{dz}} 
\def\End{{\rm End}}              

\def\sing{{\rm Sing}}            
\def\cha{{\rm char}}             
\def\Gal{{\rm Gal}}              
\def\jacob{{\rm jacob}}          
\def\tjurina{{\rm tjurina}}      
\newcommand\Pn[1]{\mathbb{P}^{#1}}   
\def\P{\mathbb{P}}
\def\Ff{\mathbb{F}}                  

\def\O{{\cal O}}                     

\def\ring{{\mathsf R}}                         
\def\R{\mathbb{R}}                   

\newcommand\ep[1]{e^{\frac{2\pi i}{#1}}}
\newcommand\HH[2]{H^{#2}(#1)}        
\def\Mat{{\rm Mat}}              
\newcommand{\mat}[4]{
     \begin{pmatrix}
            #1 & #2 \\
            #3 & #4
       \end{pmatrix}
    }                                
\newcommand{\matt}[2]{
     \begin{pmatrix}                 
            #1   \\
            #2
       \end{pmatrix}
    }
\def\cl{{\rm cl}}                

\def\hc{{\mathsf H}}                 
\def\Hb{{\cal H}}                    
\def\pese{{\sf P}}                  

\def\PP{\tilde{\cal P}}              
\def\K{{\mathbb K}}                  

\def\M{{\cal M}}
\def\RR{{\cal R}}
\newcommand\Hi[1]{\mathbb{P}^{#1}_\infty}
\def\pt{\mathbb{C}[t]}               
\def\gr{{\rm Gr}}                
\def\Im{{\rm Im}}                
\def\Re{{\rm Re}}                
\def\depth{{\rm depth}}
\newcommand\SL[2]{{\rm SL}(#1, #2)}    
\newcommand\PSL[2]{{\rm PSL}(#1, #2)}  
\def\Resi{{\rm Resi}}              

\def\L{{\cal L}}                     
\def\Aut{{\rm Aut}}              
\def\any{R}                          
\newcommand\ovl[1]{\overline{#1}}    

\newcommand\mf[2]{{M}^{#1}_{#2}}     
\newcommand\mfn[2]{{\tilde M}^{#1}_{#2}}     

\newcommand\bn[2]{\binom{#1}{#2}}    
\def\ja{{\rm j}}                 
\def\Sc{\mathsf{S}}                  
\newcommand\es[1]{g_{#1}}            
\newcommand\V{{\mathsf V}}           
\newcommand\WW{{\mathsf W}}          
\newcommand\Ss{{\cal O}}             
\def\rank{{\rm rank}}                
\def\Dif{{\cal D}}                   
\def\gcd{{\rm gcd}}                  
\def\zedi{{\rm ZD}}                  
\def\BM{{\mathsf H}}                 
\def\plf{{\sf pl}}                             
\def\sgn{{\rm sgn}}                      
\def\diag{{\rm diag}}                   
\def\hodge{{\rm Hodge}}
\def\HF{{ F}}                                
\def\WF{{ W}}                               
\def\HV{{\sf HV}}                                
\def\pol{{\rm pole}}                               
\def\bafi{{\sf r}}
\def\id{{\rm id}}                               
\def\gms{{\sf M}}                           
\def\Iso{{\rm Iso}}                           

\def\hl{{\rm L}}    
\def\imF{{\rm F}}
\def\imG{{\rm G}}

\def\TS{{\mathlarger{\mathlarger{\bf T}}}}
\def\rootsG{{\sf G}}
\def\NLL{{\rm NL}}   
\def\codmod{{\xi}}    
\def\P{{\mathbb P}}
\def\mydim{{\sf H}}    
\def\intdim{{\sf K}}   

\def\cf{r}   

\begin{center}
{\LARGE\bf  Why should one compute periods of algebraic cycles? 
}
\\
\vspace{.25in} {\large {\sc Hossein Movasati}}\footnote{
Instituto de Matem\'atica Pura e Aplicada, IMPA, Estrada Dona Castorina, 110, 22460-320, Rio de Janeiro, RJ, Brazil,
{\tt www.impa.br/$\sim$ hossein, hossein@impa.br}}
\end{center}

\begin{abstract}
In this article we show how the data of integrals of algebraic differential forms over algebraic cycles can be used
in order to prove that algebraic and Hodge cycle deformations of a given algebraic cycle are equivalent.   
As an example, we  prove that most of the Hodge and algebraic cycles of the Fermat sextic fourfold 
cannot be deformed in the underlying parameter space.  
We then take a difference of two  linear cycles inside the Fermat variety with intersection of codimension two in both cycles, 
and  gather  evidences 
that the Hodge locus corresponding to  this is smooth and reduced. This implies 
the existence of new algebraic cycles in the Fermat variety whose existence is predicted by 
the Hodge conjecture for all hypersurfaces, but not the Fermat variety itself. 
\end{abstract}

\begin{minipage}{6.08in}
\raggedleft\it\small 
... computer assisted proofs, as well as computer \\
\hfil unassisted ones, can be good or bad. A good proof \hfil \\
\hfil is a proof that makes us wiser,  (Y. Manin). \\
\end{minipage}
\section{Introduction}
A quick answer to the question of the title is the following: if we compute such numbers, put them inside a certain matrix  
and compute its rank, then either we will be able to verify the Hodge conjecture for deformed Hodge cycles, or more interestingly, 
we will find a right place to
look for counterexamples for the Hodge conjecture. 
In direction of the second situation, we collect evidences to Conjecture \ref{CanISayFinally?},  and for the first 
situation we prove Theorem \ref{27.03.17}.  
In the present text all homologies with $\Z$ coefficients
are up to torsion and all varieties are defined over complex numbers. Let $n$ be an even number. For an integer 
$-1\leq m\leq \frac{n}{2}$ let $\P^{\frac{n}{2}},\check\P^{\frac{n}{2}}\subset \P^{n+1}$ be
projective spaces given by:
\begin{equation}
\label{CarolinePilar}
 \P^{\frac{n}{2}}:  
\left\{
 \begin{array}{l}
 x_{0}-\zeta_{2d}^{}x_{1}=0,\\
 x_{2}-\zeta_{2d}^{} x_{3}=0,\\
 x_{4}-\zeta_{2d}^{} x_{5}=0,\\
 \cdots \\
 x_{n}-\zeta_{2d}^{} x_{{n+1}}=0.
 \end{array}
 \right.\ \ \ \ \ \ \ 
 \check\P^{\frac{n}{2}}:  
\left\{
 \begin{array}{l}
 x_{0}-\zeta_{2d}^{}x_{1}=0,\\
  \cdots \\
 x_{2m}-\zeta_{2d}^{} x_{2m+1}=0,\\
 x_{2m+2}-\zeta_{2d}^{3} x_{2m+3}=0,\\
 \cdots \\
 x_{n}-\zeta_{2d}^{3} x_{{n+1}}=0.
 \end{array}
 \right. 
 \end{equation}
 where $\zeta_{2d}:=e^{\frac{2\pi \sqrt{-1}}{2d}}$. 
These are linear algebraic cycles in the Fermat variety $X^d_{n}\subset \Pn {n+1}$ given  
by the homogeneous polynomial $x_0^d+x_1^d+\cdots+x_{n+1}^d=0$, and satisfy $\P^\frac{n}{2}\cap \check\P^\frac{n}{2}=\P^m$. 
By convention $\Pn {-1}$ means the empty set.
In general we can take arbitrary 
linear cycles in the Fermat variety, see \eqref{6apr2017}.  
\begin{conj}
\label{CanISayFinally?}
 Let $n\geq 6$ be an even number, $m:=\frac{n}{2}-2$ and let $\P^\frac{n}{2}$ and $\check\P^\frac{n}{2}$ be two linear cycles 
with $\P^\frac{n}{2}\cap \check\P^\frac{n}{2}=\P^m$ 
inside  the  Fermat variety of degree  $d>\frac{2(n+1)}{n-2}$,  
and let $\plc$ be the intersection 
of a linear $\Pn {\frac{n}{2}+1}\subset\Pn {n+1}$ with $X^d_n$. There is a finite, nonempty set of pairs $(\cf,\check\cf)$ 
of coprime integers  with the following property: there exists a semi-irreducible algebraic cycle $Z$ of dimension 
$\frac{n}{2}$ in  $X^d_{n}$  such that 
\begin{enumerate}
\item
For some $a,b\in\Z,\ a\not=0$, 
the algebraic cycle $Z$ is homologous  to $a(\cf\P^\frac{n}{2}+\check\cf \check\P^\frac{n}{2})+b\plc$.
\item
The deformation space of   the pair $(X^d_n,Z)$, as an analytic variety,  
contains the intersection of deformation spaces of $(X^d_n, \P^\frac{n}{2})$ and 
$(X^d_n,\check\P^\frac{n}{2})$  as a proper subset.
\end{enumerate}
\end{conj}
An algebraic cycle $Z=\sum_{i=1}^r n_i Z_i,\ \ n_i\in\Z$ in a smooth projective variety $X$ is called semi-irreducible if the pair $(X,Z)$ can be
deformed into $(X_t,Z_t)$ with $Z_t$ irreducible,  for a precise definition see \S\ref{17agu2017}. 
Note that $Z,a,b$ in the above conjecture depend on $\cf$ and $\check\cf$. 
If $d$ is a prime number or $d=4$ or $d$ is relatively prime with $(n+1)!$ then 
the Hodge conjecture for the Fermat variety $X^d_n$  
can be proved using only linear cycles, see  \cite{Ran1980} and \cite{sh79}. Therefore, the existence of the algebraic cycle $Z$
in Conjecture \ref{CanISayFinally?} is not predicted by the Hodge conjecture for $X^d_n$. We have derived it assuming the Hodge
conjecture for all smooth hypersurfaces of degree $d$ and dimension $n$ and few other 
conjectures with some computational evidences 
(Conjectures \ref{Duisburg2016},  Conjecture \ref{TrainFreiburgBonn} and Conjecture \ref{13jan2017}). 
The number $a$ is equal to $1$ if the integral Hodge conjecture is true and the term $b\plc$ pops up because the relevant 
computations are done in primitive (co)homologies.
Since the algebraic cycle $Z$
is numerically equivalent to $a(\cf\P^\frac{n}{2}+\check\cf\check\P^\frac{n}{2})+b\plc$
this might be used to investigate its (non-)existence, at least for Fermat cubic tenfold. Our computations in this article suggest that
$(\cf,\check \cf)=(1,-1)$ satisfies the property in Conjecture \ref{CanISayFinally?}. 

Let $\C[x]_d=\C[x_0,x_1,\cdots,x_{n+1}]_d$ be the set of  homogeneous polynomials of degree $d$ in $n+2$ variables, and 
let $\T$ be the open subset of $\C[x]_d$ parameterizing smooth hypersurfaces $X$ of degree $d$ and $\T_{\underline{1}}\subset \T$ be its subset 
parameterizing those with a linear $\P^{\frac{n}{2}}$ inside $X$. 
We use the notation $X_t,\ t\in\T$ and denote by $0\in\T$ the point corresponding to the Fermat variety, and so, $X_0=X^d_n$. 
The algebraic variety $\T_{\underline 1}$ is irreducible, however, as an analytic variety in a neighborhood (usual topology) of 
$0\in \T$ it has many irreducible components corresponding to deformations of 
a linear cycle  inside $X^d_n$. Let us denote by $V_{\P^{\frac{n}{2}}}$ the local branch of 
$\T_{\underline 1}$ parameterizing
deformations of the pair $\P^{\frac{n}{2}}\subset X^d_n$.  In general, 
for a Hodge cycle in $H_n(X^d_n,\Z)$ we define the Hodge locus $V_{\delta_0}\subset (\T,0)$ which is an analytic 
scheme and its underlying  
analytic variety consists of points  $t\in(\T,0)$ such that the monodromy $\delta_t\in H_n(X_t,\Z)$ of $\delta_0$ along a path 
in $(\T,0)$ is still Hodge,
see \S\ref{15july2016}.  
For $[\P^{\frac{n}{2}}]\in H_n(X^d_n,\Z)$ we know that 
$V_{[\P^{\frac{n}{2}}]}$ as analytic scheme is smooth and reduced and moreover 
$V_{\P^\frac{n}{2}}= V_{[\P^\frac{n}{2}]}$,  see the discussion after Theorem \ref{main2}. 
This is  not true for an arbitrary Hodge
cycle. 
Conjecture \ref{CanISayFinally?} says that $V_{\P^{\frac{n}{2}}}\cap V_{\check\P^{\frac{n}{2}}}$ is a proper 
subset of the Hodge locus $V_{\cf [\P^{\frac{n}{2}}]+\check\cf[\check\P^{\frac{n}{2}}]}$, see Figure \ref{TwoP1-1}.  
In Conjecture \ref{CanISayFinally?} the case $m=\frac{n}{2}-1$ and 
$\cf=\check\cf=1$  is excluded,   as the pair $(X^d_n,\P^\frac{n}{2}+\check\P^\frac{n}{2})$ can be 
deformed into  a  hypersurface containing a complete intersection 
of type $\underbrace{1,1,\cdots,1}_{\frac{n}{2} \hbox{ times }},2$. 
For small $m$'s, the situation is not also strange.
\begin{theo}[\cite{Roberto}]
\label{27.03.17}
Let $(n,d,m)$ be one of the following triples
\begin{eqnarray*}
& &(2,d,-1),\ 5\leq d\leq 14, \\
& & (4, 4,-1), (4,5,-1), (4,6,-1), (4,5,0),  (4,6,0), \\
& & (6,3, -1), (6,4, -1), (6,4,0), \\ 
& & (8,3,-1),(8,3,0),\\
& & (10,3,-1), (10,3,0), (10,3,1), 
\end{eqnarray*}
and $\P^{\frac{n}{2}}$ and $\check\P^{\frac{n}{2}}$ be linear cycles in \eqref{CarolinePilar}. 
 The Hodge locus passing through the Fermat point $0\in\T$ and corresponding to deformations of 
 the Hodge cycle $\cf[\P^{\frac{n}{2}}]+ \check\cf[\check \P^{\frac{n}{2}}]\in H_n(X^d_n,\Z)$ with 
 $\P^\frac{n}{2}\cap \check\P^\frac{n}{2}=\P^m$ and $\cf,\check\cf\in\Z,\ \cf\not=0,\ \check\cf\not=0$
 is smooth and reduced. Moreover, its 
 underlying analytic variety is simply the intersection $V_{\P^{\frac{n}{2}}}\cap V_{\check \P^{\frac{n}{2}}}$. 
\end{theo}
The cases $(n,d)=(2,4),(4,3)$ are the only cases such that the $(\frac{n}{2}+1,\frac{n}{2}-1)$ 
Hodge number of $X^d_n$ is equal to one, and these are out of our discussion as all Hodge loci $V_{\delta_0}$ 
are of codimension one, smooth and reduced. For the discussion of these cases and a baby version
of Conjecture \ref{CanISayFinally?} see \S\ref{17agu2017}.  
We conjecture that for a fixed $n\geq 6$ and $d>\frac{2(n+1)}{n-2}$, 
there is $0\leq M_{n,d}< \frac{n}{2}-2$  
depending only on $n$ and $d$ such that for $m\leq M_{n,d}$, respectively $M_{n,d}<m<\frac{n}{2}-1$, 
we have similar  statements as in Theorem \ref{27.03.17}, respectively  Conjecture \ref{CanISayFinally?}.
We do not have any idea how to describe $M_{n,d}$ in general.  
We expect that  Theorem \ref{27.03.17} for $m=-1$ is always true. In this case $\P^\frac{n}{2}$
and $\check\P^\frac{n}{2}$ do not intersect each other.  
The restriction on $n$ and $d$ in  Theorem 
\ref{27.03.17} is due to the fact that our proof is computer assisted, and upon a better computer 
programing  and a better device, it might be improved. 
For now, the author does not see any theoretical proof. 
\begin{figure}
\begin{center}
\includegraphics[width=0.6\textwidth]{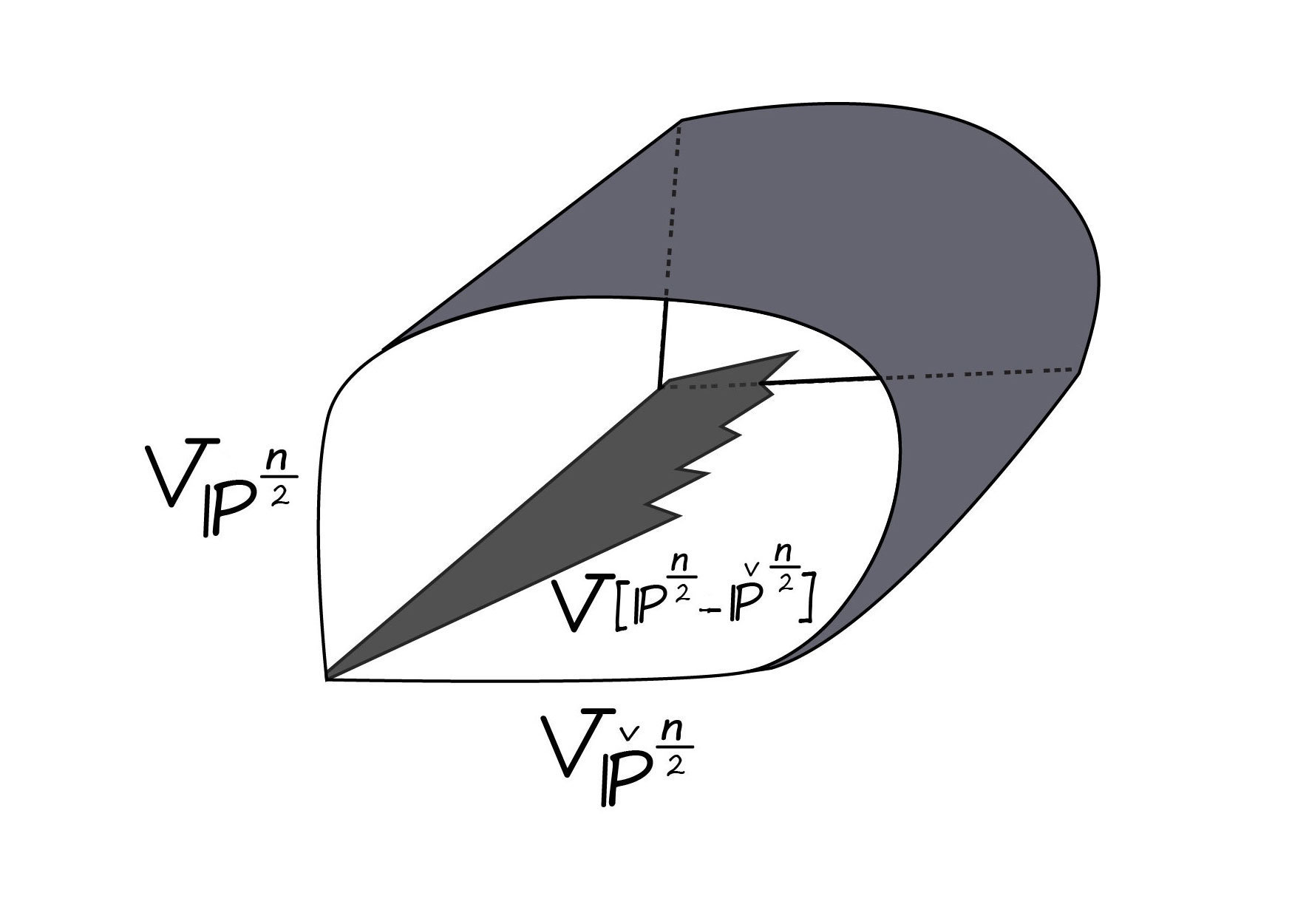}
\caption{Hodge locus of sum of linear cycles}
\label{TwoP1-1}
\end{center}
\end{figure}
The first evidence for Conjecture \ref{CanISayFinally?} is the fact that for many examples of $n$ and $d$, the codimension
of the Zariski tangent space of the analytic scheme $V_{\cf[\P^{\frac{n}{2}}]+\check\cf[\check\P^{\frac{n}{2}}] }$ is strictly smaller than the codimension of 
$V_{\P^{\frac{n}{2}}}\cap  V_{\check\P^{\frac{n}{2}}}$ which is smooth. In order to be able to investigate the smoothness and reducedness of
this analytic scheme, we have worked out Theorem \ref{InLabelNadasht?} which is just computing  a Taylor series. Its importance
must not be underestimated. The linear part of such Taylor series encode the whole data of infinitesimal variation of Hodge structures (IVHS)
introduced by Griffiths and his coauthors in 1980's, and from this one can derive most of  the applications of IVHS, such as global 
Torelli problem, see \cite{CarlsonGriffiths1980}. In particular, the proof of Theorem \ref{27.03.17} uses just such linear parts.   
In a personal communication C. Voisin pointed out the difficulties   
on higher order approximation of the Noether-Lefschetz locus. This motivated the author to elaborate some of his old ideas 
in \cite{ho06-1} and develop it into Theorem \ref{InLabelNadasht?}. The second order approximations in 
cohomological terms 
(similar to IVHS), has been formulated in \cite{mclean2005}, however it is not enough 
for the investigation of Conjecture \ref{CanISayFinally?}, see Theorem \ref{WillCome2017}, 
and it turns out one has to deal with third and 
fourth order approximations, see Theorem \ref{dream2017}.  
We use Theorem \ref{InLabelNadasht?} to check reducedness and smoothness of components of the Hodge loci.  
We break the property of being reduced and smooth  into $N$-smooth for all $N\in\N$, see \S\ref{aima2017}, 
and prove the following theorem which is not covered in Theorem \ref{27.03.17}. 
\begin{theo}
\label{WillCome2017} 
Let $(n,d,m)$ be one of the triples 
\begin{eqnarray}
& &
\label{16.06.2017-1}
(6,3,1), (6,3,0),(8,3,1)\\ \label{16.06.2017-2}
& & (4,4,0), (8,3,2),(8,3,1),(10,3,3), (10,3,2),
\end{eqnarray}
and $\P^{\frac{n}{2}}$ and $\check\P^{\frac{n}{2}}$ be linear cycles in \eqref{CarolinePilar}.
For all $\cf,\check\cf\in\Z$ with $1\leq |\cf|\leq |\check\cf|\leq 10$ 
the analytic scheme  $V_{\cf [\P^{\frac{n}{2}}]+
\check\cf[\check\P^{\frac{n}{2}}]}$ with 
$\P^{\frac{n}{2}}\cap \check \P^{\frac{n}{2}}=\P^m$ 
is $2$-smooth.  It is  $3$-smooth in the cases \eqref{16.06.2017-1} and  
for $(n,d,m,\cf,\check\cf)=(4,4,0,1,-1)$. It is $4$-smooth in the case $(n,d,m,\cf,\check\cf)=(6,3,1,1,-1)$ and $(n,d,m)=(6,3,0)$. 
\end{theo}
Note that the triples  in Theorem \ref{WillCome2017} are not covered in Theorem \ref{27.03.17} and we do not know 
the corresponding Hodge locus. 
In order to solve Conjecture \ref{CanISayFinally?} we will need to identify 
non-reduced Hodge loci. We prove that:
\begin{theo}
\label{dream2017}
Let  $\P^{\frac{n}{2}}$ and $\check\P^{\frac{n}{2}}$ be linear cycles in \eqref{CarolinePilar} 
with $\P^{\frac{n}{2}}\cap \check\P^{\frac{n}{2}}=\P^m$. 
The analytic scheme $V_{\cf[\P^{\frac{n}{2}}]+\check\cf[\check\P^{\frac{n}{2}}]}$ 
is either singular at the Fermat point 
$0$ or it is non-reduced, in the following cases:
\begin{enumerate}
\item
For all $\cf,\check\cf\in\Z,\ 1\leq |\cf|\leq |\check\cf|\leq 10,\ \cf\not=\check\cf$, $m=\frac{n}{2}-1$ and $(n,d)$ in the list
\begin{eqnarray}\label{5.12.2017}
& &(2,d),\ 5\leq d\leq 9, \\
& & (4, 4), (4,5), (6,3),(8,3) 
\end{eqnarray}
\item
For  all $\cf,\check\cf\in\Z,\ 1\leq |\cf|\leq|\check\cf|\leq 10,\ \cf\not=-\check\cf$ and $(n,d,m)$ in the list
$$
(4,4,0), (6,3,1), (8,3,2)
$$
\end{enumerate}
\end{theo}
The upper bounds for $|\cf|$ and $|\check\cf|$ is due to our computational methods, and it would not be difficult to
remove this hypothesis. The verification of the case $(n,d,m)=(8,3,2)$ in the second item  by a computer takes more than 14 days! 
Theorem \ref{dream2017} in the case  $(n,d,m)=(2,5,0)$ and without the upper bound on $|\cf|,|\check\cf|$  
follows from a theorem of Voisin in \cite{voisin89},  see   Exercise 2, 
page 154 \cite{vo03} and its reproduction in \cite{ho13} Exercise 16.9.  Based on Theorem \ref{WillCome2017} and Theorem
\ref{dream2017} we may conjecture that for $(n,d,m,\cf,\check\cf)=(4,4,0,1,-1), (6,3,1,1,-1)$, the analytic scheme 
$V_{\cf[\P^{\frac{n}{2}}]+\check\cf[\check\P^{\frac{n}{2}}]}$ is smooth
and reduced. If this is the case, its underlying analytic variety is bigger than 
$V_{[\P^{\frac{n}{2}}]}\cap V_{[\check\P^{\frac{n}{2}}]}$ 
(see \S\ref{eid1396}),  and so,
we may try to formulate similar statements as in Conjecture \ref{CanISayFinally?} in these cases. However, one of the main
ingredients of  Conjecture \ref{CanISayFinally?} fails to be true in lower degrees, see Conjecture \ref{Duisburg2016} and comments after this.

The present article together with the book \cite{ho13}
is written during the years 2014-2017. 
One of the main aims of the book \cite{ho13} has been to focus on computational aspects of Hodge theory. 
From this book we have just collected few results relevant to the content of this article, and in particular 
the study of the components of  the Hodge locus passing through the Fermat point.  
The proof of 
Theorem \ref{main1.5}, Theorem \ref{main2}, Theorem \ref{14out2016}, Theorem \ref{chilenos2017} and  
Theorem \ref{InLabelNadasht?}  are theoretical, whereas the proof of 
  Theorem \ref{27.03.17}, Theorem \ref{WillCome2017}, Theorem \ref{dream2017}, 
  Theorem \ref{28jan2016cmsa},   Theorem \ref{isolatedlikeAC}, Theorem \ref{3feb2016} are computer assisted.  These are partial verifications of many  conjectures, for which we have to work with  
particular examples of $d$ and $n$. In many cases we have just mentioned these as comments after each
conjecture and have avoided producing more theorem-style statements. 
An undergraduate student in mathematics interested in challenging problems is invited to read conjectures in 
\S\ref{25.03.2017}. 
We have to confess that  we have not done our best to verify such conjectures as much as the computer performs the computations, 
and have contented  ourselves to few special cases. 
There are few other results in the book \cite{ho13}  
which are not announced here, and they might be useful for the investigation of
Conjecture \ref{CanISayFinally?}.

The computer codes used in the present text are written as procedures in the library {\tt foliation.lib} (version 2.20) of 
{\sc Singular}, see \cite{GPS01}. The reader who wants to get used to  them  
is referred to \cite{ho13} Chapter 18. This is mainly for codes used utill 
\S\ref{eid1396}. From this section on, the name of procedures 
appears in the foot note of  the pages where they are used.  
A different computer implementation of  the proofs would be essential 
for two main reasons: first, it will be another confirmation of the results of the present paper, second, it will
produce more results that the author was not able to obtain by his own primitive codes. This may produce  precise 
conjectures for arbitrary dimension $n$ and degree $d$. 

The organization of the text is as follows. 
Sections \ref{07.03.2017-1},\ref{07.03.2017-2},\ref{15july2016},\ref{07.03.2017-3}  are essentially the
first version of the article which appeared in the Arxiv in 2015. These are  
the announcement of some of the author's results in the book  \cite{ho13}.
In \S\ref{07.03.2017-1} we reformulate the Hodge conjecture using
integrals. In \S\ref{07.03.2017-2} we introduce an alternative Hodge conjecture. This compares the deformation space of both algebraic
and Hodge cycles. In \S\ref{15july2016} we recall the missing ingredient in the formulation of infinitesimal variation of Hodge structures.
This is namely periods of Hodge/algebraic cycles. We then relate it to the alternative Hodge conjecture. In \S\ref{07.03.2017-3} we focus on 
Hodge cycles in the Fermat variety which cannot be deformed to nearby hypersurfaces. We then present  the formula of periods of linear
cycles inside the Fermat variety. 
From \S\ref{eid1396} we start to examine Conjecture
\ref{CanISayFinally?}.  In this section
we also prove Theorem \ref{27.03.17}. 
We first observe that the Zariski tangent space of the Hodge locus 
corresponding to the Hodge cycle $[\P^\frac{n}{2}+\check\P^{\frac{n}{2}}]$ has codimension strictly less  than 
the codimension of the locus corresponding to deformations of the algebraic cycle $\P^\frac{n}{2}+\check\P^{\frac{n}{2}}$. 
This indicates the existence of a strange component of the Hodge locus provided that such a component is smooth and reduced.
For this reason in \S\ref{tanhai2017} we introduce Conjecture \ref{Duisburg2016} which ensures us that such components exists for
certain linear combination of $\P^{\frac{n}{2}}$ and $\check\P^\frac{n}{2}$.
In order to investigate this conjecture,   in \S\ref{aima2017} we announce our main result on the full power series expansion of 
periods. This might be used
in order to investigate the smoothness and reducedness of the components of the Hodge loci. In this section we also prove Theorem 
\ref{WillCome2017} and Theorem \ref{dream2017}.
In \S\ref{25.03.2017} we introduce 
few other conjectures purely of linear algebraic nature. These are the last missing pieces in the proof of 
Conjecture \ref{CanISayFinally?}.  
Finally, in \S\ref{27.07.2017} we explain how to handle
Conjecture \ref{CanISayFinally?}.

My heartfelt thanks go  to P. Deligne for all his emails in January and February 2016  which
motivated me and gave me more courage and inspiration  to work on my book 
\cite{ho13} and the present article. 
This was in a time I was getting many disappointments and complains.
I would like to thank C. Voisin for her comments on higher order approximation of Noether-Lefschetz  locus.
This research has not been possible without the excellent ambient of 
my home institute IMPA in Rio de Janeiro and  the hospitality of  MPIM at Bonn during many short visits.
My sincere thanks go to both institutes. The last version of the article was written during a visit of Paris VII. I would
like to thank H. Mourtada and F. El Zein for the invitation and CNRS for financial support. Finally, I would like to dedicate this article to two women, one in my memories 
and the other by my side: Rogayeh Mollayipour, my mother, who thought me lessons
of life no other could do it, Sara Ochoa, my wife, whose contribution to the existence of
this article is not less than mine. 
\section{Hodge conjecture}
\label{07.03.2017-1}
For a complex smooth projective variety $X$, an even number $n$, an element $\omega$ of the algebraic de Rham cohomology $\omega\in H^n_\dR(X)$ 
and an irreducible subvariety  $Z$ of dimension $\frac{n}{2}$ in $X$, by a period of $Z$ we simply mean 
\begin{equation}
\label{26nov2015}
\frac{1}{(2\pi \sqrt{-1})^{\frac{n}{2}}}\mathlarger{\mathlarger{\int}}_{[Z]}\omega,
\end{equation}
where $[Z]\in H_n(X,\Z)$ is the topological class
induced by $Z$. All the homologies with integer coefficients are modulo torsions, and hence they are free $\Z$-modules.  We have to use a canonical isomorphism between the algebraic de Rham cohomology and the usual one defined by $C^\infty$-forms in order to say that
the integration makes sense, see Grothendieck's article \cite{gro66}. However, this does not give any clue how to compute such an integral.
In general, integrals are transcendental numbers, however, in our particular case 
if $X,Z,\omega$ are defined over a subfield $\k$ of complex numbers then \eqref{26nov2015} is also in $\k$, see Proposition 1.5 in 
Deligne's lecture notes in \cite{dmos}, and so  it must be computable.  
  In the $C^\infty$ context many of integrals \eqref{26nov2015} are automatically zero. This is the main content of 
the celebrated Hodge conjecture:
\begin{conj}[Hodge Conjecture]
\label{HC}
Let $X$ be a smooth projective variety of even dimension $n$ and $\delta\in H_n(X,\Z)$ be a Hodge cycle, that is, 
$$
\int_{\delta}\omega=0,\ \  \hbox{ for all closed $(p,q)$-form in $X$ with  }  p>\frac{n}{2},\ p+q=n.
$$
Then there is an algebraic cycle 
$$
\sum_{i=1}^s n_iZ_i,\ \ n_i\in\Z,\ \ \dim(Z_i)=\frac{n}{2}
$$ 
and a natural number $a\in\N$ such that $a\cdot \delta=\sum n_i[Z_i]$. 
\end{conj}
Using Poincar\'e duality our version of the Hodge conjecture is equivalent to the official one, see for instance 
Deligne's announcement of the Hodge conjecture \cite{Deligne-HodgeConjecture},  however,  we wrote it in this format in order to point out that the Hodge decomposition is not needed in 
its announcement and bring it to its origin  which is the study of integrals due to Abel, Poincar\'e, Picard among many others. 
For a prehistory of the Hodge conjecture see \cite{ho13}, Chapters 2 and 3.

\section{An alternative conjecture}
\label{07.03.2017-2}
The Hodge conjecture does not
give any information about non-vanishing integrals \eqref{26nov2015}. 
In this article we show that explicit computations of \eqref{26nov2015} 
lead us  to verifications of the  following alternative for the Hodge conjecture:
\begin{conj}[Alternative Hodge Conjecture]\rm
\label{IHC}
Let $\{X_t\}_{t\in \T}$ be a family of complex smooth projective varieties of even dimension $n$,  
and let  $Z_0$ be a fixed irreducible algebraic cycle of dimension 
$\frac{n}{2}$ in $X_0$ for $0\in\T$.  
There is an open neighborhood $U$ of $0$ in $\T$ (in the usual topology) such that for all $t\in U$ if 
the monodromy $\delta_t\in H_n(X_t,\Z)$ of $\delta_0=[Z_0]$ is a Hodge cycle,  then there is an algebraic deformation $Z_t\subset X_t$ of 
$Z_0\subset X_0$ such that $\delta_t=[Z_t]$. In other words, deformations of $Z_0$ as a Hodge cycle and as an algebraic cycle 
are the same. 
\end{conj}
Before explaining the relation of this conjecture with integrals \eqref{26nov2015}, 
we say few words about the importance of Conjecture \ref{IHC}.
First of all, Conjecture \ref{IHC}  might be false in general, therefore, it might be called a property of $Z_0$. 
P. Deligne  pointed out that there are additional obstructions to the hope that algebraic cycles could 
be constructed by deformation (personal communication, 31 January 2016).  For instance, the dimension of the intermediate Jacobian coming from the largest sub Hodge structure of  
$H^{n-1}(X_0,\Q)\cap (H^{\frac{n}{2},\frac{n}{2}-1}\oplus H^{\frac{n}{2},\frac{n}{2}-1})$  might jump down by deformation.
This observation does not apply to a smooth hypersurface, for which only the middle cohomology is non-trivial. 
We are interested in cases in which Conjecture \ref{IHC} is true, see Theorem \ref{main1} below.  
Both Hodge conjecture  and Conjecture \ref{IHC} claim that a given  Hodge cycle must be algebraic, however, note that 
Conjecture \ref{IHC} provides a candidate for such an algebraic cycle, whereas the Hodge conjecture doesn't,  and so,  it must be easier than the
Hodge conjecture.  Verifications of Conjecture \ref{IHC} support  the Hodge conjecture, however, a counterexample to Conjecture \ref{IHC} might not 
be a counterexample to the Hodge conjecture, because one may have an algebraic cycle homologous to, but 
different from, the  given one in Conjecture \ref{IHC}. 

In \cite{gro66} page 103 Grothendieck states a conjecture which is as follows: 
let $X\to S$ be a smooth morphism of schemes and let $S$ be  connected and reduced. A global section 
$\alpha$ of  $H^{2p}_\dR(X/S)$ is algebraic  at every fiber $s\in S$
if and only if it is a flat section with respect to the Gauss-Manin connection and  
it is algebraic for one point $s\in S$. Conjecture \ref{IHC}, for instance for complete
intersections inside hypersurfaces,   implies this conjecture in the same context, 
however the vice versa is not true. The variety $\T_{\underline{d}}$ defined in \S\ref{15july2016} 
might be a proper subset of a component of 
the Hodge locus. This would imply that $Z$ is homologous to another algebraic cycle 
with a bigger deformation space. This cannot happen for the linear case 
$\underline{d}=(1,1,\cdots,1)$, see Theorem \ref{main1} below, 
and many examples of $n$ and $d$ and $\underline{d}$,  see \cite{Roberto}. 
The  article \cite{Bloch1972} is built upon the Grothendieck's conjecture explained above and
it considers semi-regular algebraic cycles, that is, the semi-regularity map
$\pi: H^1(Z,N_{X/Z})\to H^{\frac{n}{2}+1}(X,\Omega^{\frac{n}{2}-1})$ is injective. 
The semi-regularity is a very strong  condition. For instance, 
for curves inside surfaces, \cite{Bloch1972} only considers the semi-regular curves with 
$H^1(Z,N_{X/Z})=0$. Using Serre duality, one can easily see that 
this is not satisfied for curves with self intersection less than $2g-2$, where $g$ is 
the genus of $Z$. A simple application of adjunction formula shows that apart from few cases,
complete intersection  curves inside surfaces do not satisfy this condition.

In situations where the Hodge conjecture
is true, for instance for surfaces, Conjecture \ref{IHC} is still a non-trivial statement. 
For a smooth hypersurface $X\subset \Pn 3$  of degree $d\geq 4$ and 
a line $\Pn 1\subset X$, deformations of $\Pn 1$ as a Hodge cycle  and as an algebraic curve
are the same. This follows from classical IVHS techniques introduced in \cite{CGGH1983}. 
In \cite{green1988, green1989}, \cite{voisin1988}, Green and Voisin prove a stronger statement which says that the space of surfaces $X\subset \Pn 3$ containing a line $\Pn 1$ is
the only component of the Noether-Lefschetz locus of minimum codimension $d-3$. In order to reproduce the full statement of Green and
Voisin's results in our context and in a neighborhood of the Fermat point, see Conjecture \ref{1apr2017} and the comments after. 
 In a similar way  some other results of Voisin on Noether-Lefschetz loci, see \cite{voisin90},
 fit into the framework of Conjecture \ref{IHC}. 
 A weaker version of the mentioned statement  in higher dimensions is generalized in the following way:
 \begin{theo}[\cite{GMCD-NL} Theorem 2]
 \label{main1}
 For any smooth hypersurface $X$ of degree $d$ and dimension $n$ in a Zariski neighborhood of the Fermat variety 
 with $d\geq 2+\frac{4}{n}$ and a linear projective space $\Pn {\frac{n}{2}}\subset X$, 
deformations of $\Pn {\frac{n}{2}}$ as an algebraic cycle and Hodge cycle are the same. 
\end{theo} 

\section{Infinitesimal variation of Hodge structures for Fermat variety}
\label{15july2016}
The relation between integrals \eqref{26nov2015} and Conjecture \ref{IHC} is established through the so-called
infinitesimal variation of Hodge structures developed in \cite{CGGH1983}. 
This is explained in \cite{GMCD-NL}, where the author 
has tried to keep the classical language of IVHS, and so we do not reproduce it here. The main application is going to be 
on Hodge and Noether-Lefschetz loci. 
The reader is referred to 
Voisin's expository article \cite{voisinHL} which contains a full exposition and main references on this topic.

In order to keep the content of this text elementary, we explain
this for complete intersection algebraic cycles  inside hypersurfaces, and in particular, the Fermat variety.
Let $\T$ be the parameter space of smooth hypersurfaces of degree $d$ in $\Pn {n+1}$. A hypersurface $X=X_t,\ t\in \T$ 
is  given by the projectivization  of $f(x_0,x_1,\cdots, x_{n+1})=0$, where $f$ is a homogeneous polynomial of degree $d$.  
Fix integers $1\leq d_1,d_2,\ldots, d_{\frac{n}{2}+1}\leq d$ and $\underline d :=(d_1,d_2,\ldots, d_{\frac{n}{2}+1})$.
Let   $\T_{\underline d}\subset \T$ 
be the parameter space of smooth hypersurfaces with
$$
f=f_1f_{\frac{n}{2}+2}+\cdots+ f_{\frac{n}{2}+1}f_{n+2},\ \ \deg(f_i)=d_i,\ \deg(f_{\frac{n}{2}+1+i})=d-d_i,
$$
where $f_i$'s are homogeneous polynomials. 
The algebraic cycle
\begin{equation}
 \label{golnargolnar}
Z:=\Pn {}\{f_1=f_2=\cdots=f_{\frac{n}{2}+1}=0\}\subset X
\end{equation}
is called a complete intersection (of type $\underline d$) in $X$.  Note that this cycle is a complete intersection in $\Pn {n+1}$
and it is not a complete intersection of $X$ with other hypersurfaces.  
Let 
\begin{equation}
\label{11.06.2017}
 \omega_i:=\Resi
 \left( 
\frac{x^i\cdot \sum_{j=0}^{n+1}(-1)^j x_j\ dx_0\wedge \cdots \wedge \widehat dx_j\wedge \cdots\wedge dx_{n+1}}
     {f^{k}}
\right)
\end{equation}
with $k:=\frac{n+2+\sum_{e=0}^{n+1}i_e}{d}$, where $\Resi : H^{n+1}_\dR(\Pn {n+1}-X)\to H^n_\dR(X)$ 
is the residue map and $x^i=x_0^{i_0}\cdots x_{n+1}^{i_{n+1}}$. 
After Griffiths \cite{gr69}, we know that $\delta\in H_n(X,\Z)$ is a Hodge cycle if and only
if 
\begin{equation}
\label{8.4.2017}
\int_{\delta}\omega_i=0,\ \ \forall i \ \ \ \hbox{  with  }\ \ \ \   \frac{n+2+\sum_{e=0}^{n+1}i_e}{d}\leq \frac{n}{2}. 
\end{equation}
A cycle $\delta\in H_n(X,\Z)$ is called primitive if its intersection with $[\plc]$ is zero. Recall
that $\plc$ is  the intersection of a linear $\Pn {\frac{n}{2}+1}\subset \Pn {n+1}$ with $X$. The $\Z$-module 
$H_n(X,\Z)_\prim$ by definition is the set of primitive cycles. 
We denote by $\hodge_n(X,\Z)\subset H_n(X,\Z)$ the $\Z$-modules of $n$-dimensional Hodge
cycles in $X$, and by $\hodge_n(X,\Z)_\prim$ its submodule consisting of primitive cycles. 
All the $\Z$-modules in this text are up to torsions, and hence they are free. 

Let us now focus on the Fermat variety $X^d_n$ which is obtained by the projectivization of 
\begin{equation}
\label{27nov2015}
X^d_n:   \ \  x_0^{d}+x_1^{d}+\cdots+x_{n+1}^d=0. 
\end{equation}
We denote   by $0\in \T$ the point corresponding to $X^d_n$, that is,  $X_0=X^d_n$.
Hodge cycles of the Fermat variety have been extensively studied by Shioda in his  seminal works 
\cite{Shioda1979, sh79,  Shioda1981}.
We are mainly interested in the Hodge cycles $[Z]$, where $Z$ is a complete intersection
of type  $\underline{d}$ in $\Pn {n+1}$ which lies in  
$X^d_n$. This is because all the examples of $n$ and $d$ in which the Hodge conjecture is known for $X^d_n$, 
one has only used this type of algebraic cycles, see \cite{ho13} Chapter 17.  
The periods of a Hodge cycle $\delta\in \hodge_n(X^d_n,\Z)$ are defined
in the following way
 \begin{equation}
 \label{artan}
\pn_i=\pn_i(\delta):=\frac{1}{(2\pi \sqrt{-1})^{\frac{n}{2}}}\int_{\delta}\omega_i,\ \ 
\end{equation}
$$
\sum_{e=0}^{n+1}i_e=(\frac{n}{2}+1)d-(n+2).
 $$
Using  Deligne's result in \cite{dmos} Proposition 1.5, we know that $\pn_i$'s are 
in an abelian extension of of $\Q(\zeta_{d})$. If $\pn_i$'s are all zero then $\delta$ is necessarily in the one
dimension $\Q$-vector space generated by $[\plc]$.
We are going to explain the role of these numbers in the deformation of Hodge cycles.
\begin{defi}\rm
\label{05/04/2017}
For natural numbers  \(N\), \(n\) and \(d\) let us define
\begin{equation}
\label{21oct2014}
I_N:=\left \{ (i_0,i_1,\ldots,i_{n+1})\in {\mathbb Z}^{n+2}\mid 0\leq i_e\leq d-2, \ \ i_0+i_1+\cdots+
i_{n+1}=N\right\}
\end{equation}
Assume that \(n\) is even and \(d\geq 2+\frac{4}{n}\).  Consider complex numbers \(\pn_i\)  
indexed by \(i\in I_{(\frac{n}{2}+1)d-n-2}\). 
For any other \(i\) which is not in the set \(I_{(\frac{n}{2}+1)d-n-2}\),  we define \(\pn_i\) to be  zero.
Let \([\pn_{i+j}]\)  be a matrix whose rows and columns are indexed by \(i\in I_{\frac{n}{2}d-n-2}\) and \(j\in I_d\), 
respectively, and in its \((i,j)\) entry we have \(\pn_{i+j}\). 
\end{defi}
The numbers $\#I_d, \#I_{\frac{n}{2}d-n-2},\ \#I_{(\frac{n}{2}+1)d-n-2}$ are respectively, 
the dimension of the moduli space, $(\frac{n}{2}+1,\frac{n}{2}-1)$ Hodge number and $(\frac{n}{2},\frac{n}{2})$
Hodge number minus one, of smooth hypersurfaces of dimension $n$  and degree $d$.
The following theorem justifies the importance of the algebraic numbers $\pn_i$'s in \eqref{artan}. 
\begin{theo}
\label{main1.5}
Let $X^d_n$ be the Fermat variety of dimension $n$ and degree $d$ parameterized by the point $0\in\T$.
 Let also  $\delta_0\in \hodge_n(X^d_n,\Z)$ be a Hodge cycle. 
The kernel of the matrix $[\pn_{i+j}]$ is canonically identified with the Zariski tangent space  of the 
Hodge locus $V_{\delta_0}$ passing through 
$0\in\T$ and corresponding to $\delta_0$.
\end{theo}
The Hodge locus mentioned in the above theorem is actually the analytic scheme defined by  
\begin{equation}
\label{10maio16}
\O_{V_{\delta_0}}:=
\O_{\T,0}\Bigg/\left\langle \mathlarger{\int}_{\delta_t}\omega_1, \mathlarger{\int}_{\delta_t}\omega_2, \cdots, \mathlarger{\int}_{\delta_t}\omega_a \right\rangle 
\end{equation}
where $\omega_1,\omega_2,\cdots,\omega_a$ are  reindexed $\omega_i$'s in \eqref{8.4.2017}. These are  sections of 
the cohomology bundle  $H^n_\dR(X_t),\ t\in(\T,0)$ such that for 
$t\in(\T,0)$ they form a basis of  $F^{\frac{n}{2}+1}H^n_\dR(X_t)$, where $F^i$'s are the pieces
of the Hodge filtration of $H^{n}_\dR(X_t)$. Its points are all $t$ in a small neighborhood of $0$
such that the monodromy $\delta_t\in H_n(X_t,\Z)$ of $\delta_0$ is a Hodge cycle, or equivalently, 
$\mathlarger{\int}_{\delta_t}\omega_1=\mathlarger{\int}_{\delta_t}\omega_2=\cdots=\mathlarger{\int}_{\delta_t}\omega_a=0$. 
This is a local 
analytic subset of $\T$ and by a deep theorem of 
Cattani-Deligne-Kaplan in\cite{cadeka} we know that it is algebraic. This together with the fact that Hodge cycles
of the Fermat variety are absolute and Deligne's Principle B in \cite{dmos} implies that such an algebraic set is defined 
over $\bar\Q$,  for details see 
\cite{voisinHL} Proposition 5.7.
The Hodge locus in $\T$ is the union of all such 
local loci defined as before for all $t\in\T$ (one might take different $\omega_i$'s as in \eqref{8.4.2017}). 
Theorem \ref{main1.5} follows from   Voisin's result  \cite{vo03} 5.3.3 on the Zariski tangent
space of the Hodge locus and the  computations of the infinitesimal variation of Hodge structures  for the Fermat variety in \cite{GMCD-NL}.
An alternative proof using some ideas of holomorphic foliations is given in the later reference.  
\begin{theo}
\label{main2}
Let $X^d_n$  be the Fermat variety \eqref{27nov2015} and let $Z$ be a complete intersection of type $\underline{d}$
inside $X^d_n$.  Let also $\pn_i$ be the periods of $\delta=[Z]$ defined in \eqref{artan}. 
If
  \begin{equation}
  \label{12nov2015}
 \rank([\pn_{i+j}]) =
 \bn{n+1+d}{n+1}-
 \sum_{k=1}^{n+2}(-1)^{k-1} \sum_{a_{i_1}+a_{i_2}+\cdots+a_{i_k}\leq d }\bn{n+1+d-a_{i_1}-a_{i_2}-\cdots-a_{i_k}}{n+1}
\end{equation}
where $(a_1,a_2,\ldots,a_{n+2})=(d_1,d_2,\ldots,d_{\frac{n}{2}+1},d-d_1,d-d_2,\ldots,d-d_{\frac{n}{2}+1})$
and the second sum runs through all $k$ elements (without order) of $a_i,\ \ i=1,2,\ldots,n+2$,  
then $\T_{\underline d}$ is a component of the Hodge locus. In particular 
Conjecture \ref{IHC} is true for smooth hypersurfaces $X\subset \Pn {n+1}$ containing 
a complete intersection of type $\underline{d}$,  and in a non-empty Zariski open subset of $\T_{\underline{d}}$.
 \end{theo}
The number in the right hand
side of \eqref{12nov2015} is actually the codimension of $\T_{\underline d}$  in $\T$, see \cite{ho13} Proposition 17.5, and
so Theorem \ref{main2} is a consequence of this fact and Theorem \ref{main1.5}, see \cite{ho13} Theorem 17.6. 
For arbitrary $d$ and $n$ the hypothesis of Theorem \ref{main2} is verified for projective spaces 
$Z=\Pn {\frac{n}{2}}\subset X$, that is, for the case $\underline{d}=(1,1,\ldots,1)$.  This is
 \begin{equation}
  \label{12nov2015-2016}
 \rank([\pn_{i+j}]) =
 \bn{\frac{n}{2}+d}{d}-(\frac{n}{2}+1)^2,
\end{equation}
see \cite{GMCD-NL}.  
In this way we have
derived Theorem \ref{main1}. For this particular class of algebraic cycles, it is possible to prove the
identity \eqref{12nov2015-2016} without computing $\pn_i$'s. 
 We may expect or conjecture
that the equality \eqref{12nov2015} is always true. This is the case for  many examples of 
complete intersection algebraic cycles worked out in \cite{Roberto}. This includes the author's favorite example 
$(n,d)=(4,6)$, that is, the sextic Fermat fourfold:
\begin{equation}
X^6_4: x_0^6+x_1^6+x_2^6+x_3^6+x_4^6+x_5^6=0.
\end{equation}
The Hodge numbers of the fourth cohomology of a smooth sextic fourfold is 1,426,1752, 426,1, and the Fermat sextic 
fourfold has a very peculiar property that the $\Q$-vector space of its Hodge cycles has the maximum
dimension which is 1752.  In this case the matrix $[\pn_{i+j}]$ is a quadratic $426\times 426$ matrix. 
We have only $10$ possibilities for the locus $\T_{\underline{d}}$  of hypersurfaces with a complete intersection algebraic cycle. 
The corresponding data
are listed in the table below: 
 \begin{center}
\begin{tabular}{|c|c|c|c|}
\multicolumn{4}{c}{Codimension of the loci of complete intersection algebraic cycles}
\\ 
\hline
$(d_1,d_2,d_3)$                   & $\codim_\T(\T_{\underline{d}})$ &  
$(d_1,d_2,d_3)$                   & $\codim_\T(\T_{\underline{d}})$ 
\\  \hline  
$(1,1,1)$                               & 19  &  $(1,3,3)$ &  71 \\ \hline 
$(1,1,2)$                                & 32 & $(2,2,2)$ & 92  \\ \hline
$(1,1,3)$                                & 37  & $(2,2,3)$&  106  \\ \hline
$(1,2,2)$                                &  54 &  $(2,3,3)$&  122    \\ \hline
$(1,2,3)$                                & 62   &  $(3,3,3)$ &  141  \\  \hline         
\end{tabular}
\label{digaque}
\end{center}
The Fermat cubic tenfold 
$$
X^{3}_{10}:\ x_0^3+x_1^3+\cdots+x_{11}^{3}=0
$$
has the Hodge numbers $0,0,0,1,220,925,220,1,0,0,0$ and 
the $\Q$-vector space of its Hodge cycles has the maximum dimension which is $925$.
In this case the Hodge conjecture can be verified using linear cycles $\Pn 5$,
see Theorem \ref{24.07.2016}. We have only one possibility for $\T_{\underline{d}}$. This is namely $\underline{d}=(1,1,1,1,1,1)$.
Its codimension is $20$. 

\section{General Hodge cycles for Fermat variety}
\label{07.03.2017-3}
We say that a Hodge cycle $\delta\in \hodge_n(X^d_n,\Q)$ is general if $\rank[\pn_{i+j}]$ attains the maximal rank, that is,
\begin{equation}
\label{28jan2016}
\rank[\pn_{i+j}]={\rm minimum}\{\#I_d, \# I_{\frac{n}{2}d-n-2}\}.
\end{equation}
Note that $\# I_{\frac{n}{2}d-n-2}$ (resp, $\#I_d$) is the number of rows (resp. columns) of $[\pn_{i+j}]$. 
If there exists a general
Hodge cycle then  the subvariety of  $\hodge_n(X^d_n,\Q)$ given by $\rank[\pn_{i+j}]<{\rm min}\{\#I_d, \# I_{\frac{n}{2}d-n-2}\}$ is proper and
so there is a Zariski open subset $U$ of $\hodge_n(X^d_n,\Q)$ such that all $\delta\in U$ are general. 
This will hopefully justify the name. Moreover, Theorem \ref{main1.5} implies that for 
a general Hodge cycle $\delta$ the Hodge locus $V_{\delta}$ is always smooth and reduced. 
\begin{conj}
\label{29/1/16}
The Fermat variety $X^d_n$ for $d\geq 2+\frac{4}{n}$ has always a general Hodge cycle. 
\end{conj}
Let us discuss two extreme cases in the above conjecture. First, if $n=2$ then a general Hodge cycle has 
$\rank[\pn_{i+j}]=\bn{d-1}{3}$. Conjecture \ref{29/1/16} and Theorem 
\ref{main1.5} imply that there are infinite number of components of the Noether-Lefschetz locus of codimension $\bn{d-1}{3}$ 
passing through the Fermat point $0\in \T$, provided that there are infinite number of general Hodge cycles with different 
$\ker[\pn_{i+j}]$. This is compatible with the result in \cite{CHM88} that the components of the Noether-Lefschetz locus
with the maximal codimension are dense in $\T$, both in the Zariski and usual topology. Second, if 
$d>\frac{2(n+1)}{n-2}$ then the right hand side of \eqref{28jan2016} is $\# I_d$  which is also the dimension of the moduli of hypersurfaces of degree $d$
and dimension $n$. Therefore, Conjecture \ref{29/1/16} in this case implies that general Hodge cycles of the Fermat variety
cannot be deformed in the moduli space of hypersurfaces of degree $d$ and dimension $n$, in other words, 
any deformation of a general Hodge cycle of the Fermat variety to a nearby
hypersurface $X\subset \Pn {n+1}$ implies that $X$ is obtained from $X^d_n$ by a linear transformation of $\Pn {n+1}$.


For a moment assume that we have a collection of algebraic cycles $Z_i,\ \ i=1,2,\ldots,s$ such
that $[Z_i]$'s generate the $\Q$-vector space $\hodge_n(X^d_n,\Q)$ of Hodge cycles, and so, we know 
 the Hodge conjecture for $X^d_n$ is valid. This together  with Conjecture \ref{29/1/16} implies that a general
 algebraic cycle $\sum_{i=1}^s n_i Z_i,\ n_i\in\Z$ has a deformation space of the expected codimension which is the right
 hand side of \eqref{28jan2016}. In particular, for $d>\frac{2(n+1)}{n-2}$ such an algebraic cycle cannot be deformed at all if
 we consider the parameter space $\T$ parameterizing the homogeneous polynomials of the type  
\begin{equation}
\label{khodmidanad}
f:= x_0^{d}+x_1^{d}+\cdots+x_{n+1}^d- \sum_{j \in I_d}t_j x^j .\ \ \ 
\end{equation}
Any smooth hypersurface  in a Zariski open neighborhood of the Fermat point $0$ after a linear transformation of $\Pn {n+1}$, 
can be written as the zero set of some $f$ in this format. 

The equalities \eqref{12nov2015} and \eqref{28jan2016} can be checked computationally, as far as, we take particular examples 
of the degree $d$ and the dimension $n$, compute the periods $\pn_i$ and   the rank of $[\pn_{i+j}]$. Here, is the result
\begin{theo}
\label{28jan2016cmsa}
The Fermat surface $X^d_2,\ \ 4\leq d\leq 8$ has a general Hodge cycle.
The Fermat fourfold $X^d_4,\  3\leq d\leq 6$  has also a general Hodge cycle. 
 \end{theo}  
The upper bound on $d$ is just  due to the limitation of our computer and it might be improved if one uses a better computing machine.
Theorem \ref{28jan2016cmsa} for $n=4, d=6$ says the following: 
\begin{theo}
\label{isolatedlikeAC}
A general Hodge cycle $\delta_0\in H_{4}(X^6_{4},\Q)$ is not deformable, that is,  
the monodromy $\delta_t\in H_4(X_t,\Z),\ \ t\in(\T,0)$ 
of  $\delta_0$ to $X_t$ is no more a Hodge cycle. 
\end{theo}
Note that we are using the parameter space  in \eqref{khodmidanad}, otherwise, we should have stated that $X_t$ is obtained by a 
linear transformation of $X^6_{4}$. For the computations of periods of Hodge cycles and proof of 
Theorem \ref{28jan2016cmsa} and Theorem \ref{isolatedlikeAC} see \cite{ho13} Chapter 15 and 16. See also
\S 18.8 for details of the computer codes used for the proofs. The same codes for the Fermat cubic tenfold runs out of memory. In this
case one might use Theorem \ref{14out2016}.
\begin{theo}[\cite{Ran1980}, \cite{sh79}, \cite{AokiShioda1983}]
\label{24.07.2016}
Suppose that either $d$ is a prime number or $d=4$ or $d$ is relatively prime with $(n + 1)!$.
Then $\hodge_n(X^d_n,\Q)$ is generated by the homology classes of the linear cycles $\Pn {\frac{n}{2}}$, and in particular, 
 the Hodge conjecture for $X^d_n$ is true.
\end{theo}
This theorem is the outcome of many efforts in order to prove the Hodge conjecture for $X^d_n$ using linear projective cycles. The cases  
$(n,d)=(2,6), (4,6)$ are not covered by this theorem because such algebraic cycles are not enough in these cases. 
N. Aoki in \cite{aoki1987}, 
inspired by his work with Shioda \cite{AokiShioda1983}, has introduced more algebraic cycles 
and in this way he has been able to verify the Hodge conjecture for many other Fermat varieties, and in particular for 
the sextic Fermat fourfold. In this case we can determine the homology classes of linear cycles $\Pn {\frac{n}{2}}$ explicitly, 
\cite{ho13} Section 16.7. 
This together with Theorem \ref{isolatedlikeAC} gives us:
 \begin{theo}
 \label{3feb2016}
 A general $\Z$-linear combination of projective linear cycles $\Pn {2}$  is not
 deformable in the moduli space of degree $6$ hypersurfaces in $\Pn {5}$. 
 \end{theo}

We propose two different methods in order to compute integrals \eqref{26nov2015}. 
The first method is purely
topological and it is based on the computation of the intersection numbers of algebraic cycles with 
vanishing cycles. In the case of the Fermat variety, we are able to write down vanishing 
cycles explicitly, however, they 
are singular, even though they are homeomorphic to spheres, and many interesting algebraic cycles
of the Fermat variety intersect them in their singular points. This makes the computation of intersection numbers harder.
The second method is purely algebraic and it is a generalization of Carlson-Griffiths 
computations in \cite{CarlsonGriffiths1980}.  
One has to compute the restriction of differential $n$-forms in $X$
to the top cohomology of $Z$, and then,  one has to compute the so-called trace map. The second method is 
the main topic of the Ph.D. thesis of R. Villaflor, see \cite{RobertoThesis}.  For 
$a=(a_1,a_3,\ldots,a_{n+1})\in\{0,1,2,\ldots,d-1\}^{\frac{n}{2}+1}$ and a permutation $b=(b_0,b_1,\ldots,b_{n+1})$ of $\{0,1,2,\ldots,n+1\}$ 
let 
\begin{equation}
\label{6apr2017}
 \licy{\frac{n}{2}}_{a,b}:  
\left\{
 \begin{array}{l}
 x_{b_0}-\zeta_{2d}^{1+2a_1}x_{b_1}=0,\\
 x_{b_2}-\zeta_{2d}^{1+2a_3} x_{b_3}=0,\\
 x_{b_4}-\zeta_{2d}^{1+2a_5} x_{b_5}=0,\\
 \cdots \\
 x_{b_n}-\zeta_{2d}^{1+2a_{n+1}} x_{b_{n+1}}=0.
 \end{array}
 \right. 
\end{equation}
We call it a linear cycle inside the Fermat variety. Hopefully, this $a,b$ notation will not be confused with the integers 
$a,b$ in Conjecture \ref{CanISayFinally?}. In order to avoid repetitions, we may assume that 
$b_0=0$ and for $i$ an even number $b_i$ is the smallest number in 
$\{0,1,\ldots,n+1\}\backslash\{b_0,b_1,b_2,\ldots, b_{i-1}\}$. In this way the number of linear cycles  is 
\begin{equation}
\label{28.06.2017}
(n+1)\cdot (n-1)\cdots 3\cdot 1\cdot d^{\frac{n}{2}+1}.
\end{equation}
For linear cycles 
the computation of periods is a direct consequence of a
theorem of Carlson and Griffiths in \cite{CarlsonGriffiths1980}:
\begin{theo}
\label{14out2016}
For  $i\in I_{(\frac{n}{2}+1)d-n-2}$ we have
$$
\frac{1}{(2\pi \sqrt{-1})^{\frac{n}{2}}}\mathlarger{\mathlarger{\int}}_{\licy{\frac{n}{2}}_{a,b}}\omega_i= 
$$
$$
\left\{
	\begin{array}{ll}
		   \frac{{\rm sign}(b)\cdot(-1 )^
		   {\frac{n}{2}}}{d^{\frac{n}{2}+1}\cdot \frac{n}{2}!} 
		   {\zeta_{2d}}^{\epsilon}  & \mbox{if } \ \ \  
		   \ i_{b_{2e-2}}+i_{b_{2e-1}}=d-2, \ \ \ \forall e=1,...,\frac{n}{2}+1, \\
		0 & \mbox{otherwise. } 
	\end{array}
\right.
$$
where $\zeta_{2d}$ is the $2d$-th primitive root of unity and 
$$
\epsilon=\sum_{e=0}^\frac{n}{2} (i_{b_{2e}}+1)\cdot (1+2a_{2e+1}).
$$
\end{theo}
This is done in \cite{Roberto} and it is the main ingredient of Theorem \ref{27.03.17}
Theorem \ref{main1} follows from the verification of the equality
\eqref{12nov2015-2016} for periods of linear cycles computed in Theorem \ref{14out2016}. 
This verification turns out to be an elementary problem. Using Theorem \ref{14out2016} we can make Theorem \ref{3feb2016} more concrete.

\def\codnum{{\sf C}}
\section{Sum of two linear cycles}
\label{eid1396}
Let $\P^{\frac{n}{2}}, \check\P^{\frac{n}{2}}$ be two linear algebraic cycles in the Fermat variety.  
We define
 \begin{equation}
 \label{6/12/2016}
 \mydim^{d}_n(m):= \rank\left(\left[\pn_{i+j}\left( [\P^{\frac{n}{2}}]+[\check\P^{\frac{n}{2}}]\right)\right]\right),
 \hbox{ where }  \P^{\frac{n}{2}}\cap \check \P^{\frac{n}{2}}=\P ^m.
 \end{equation} 
 \begin{conj}
 \label{colombianos2017} 
The number $\mydim^{d}_n(m)$ 
depends only on $d,n,m$ and not on the choice of $\P^{\frac{n}{2}}, \check\P^{\frac{n}{2}}$.
\end{conj}
\href{http://w3.impa.br/~hossein/WikiHossein/files/Singular%20Codes/2017-06-CodimensionIsIndependent.txt}
{We have verified the conjecture for $(n,d)$ in:}\footnote{The procedure {\tt ndm} is used for this purpose.}
$$
(2,d),\ \ 5\leq d\leq 8,\  (4,4), (6,3).  
$$
We can use the automorphism group $\group^d_n$ of the Fermat variety and we can assume that 
$\P^{\frac{n}{2}}$ is \eqref{6apr2017} 
with $a=(0,0,\cdots,0)$ and $b=(0,1,\cdots,n+1)$. In order to avoid 
Conjecture \ref{colombianos2017} we will fix our choice of linear cycles:
\begin{eqnarray*}
\P^\frac{n}{2}&=&\P^\frac{n}{2}_{a,b} \hbox{ with } a=(0,0,\cdots,0),\ b=(0,1,\cdots,n+1)\\
\check \P^\frac{n}{2}&=&\P^\frac{n}{2}_{a,b} \hbox{ with } a=(\underbrace{0,0,\cdots,0}_{m+1 \hbox{ times}},
1,1,\cdots,1),\ b=(0,1,\cdots,n+1)
\end{eqnarray*}
which are those used in Introduction. 
For examples of $\mydim^d_n(m)$ see Table \ref{12mar2017}.   
For a sequence of natural numbers $\underline a=(a_1,\ldots,a_{s})$ let
us define
\begin{equation}
 \label{1julio2016-bonn}
 \codnum_{\underline{a}}=
 \bn{n+1+d}{n+1}-
 \sum_{k=1}^{s}(-1)^{k-1} \sum_{a_{i_1}+a_{i_2}+\cdots+a_{i_k}\leq d }\bn{n+1+d-a_{i_1}-a_{i_2}-\cdots-a_{i_k}}{n+1},
\end{equation}
where the second sum runs through all $k$ elements (without order) of $a_i,\ \ i=1,2,\ldots,s$. 
By our convention, the projective space $\P^ {-1}$ means the empty set.
By abuse of notation we write 
$$
a^b:=\underbrace{a,a,\cdots,a}_{\hbox{$b$ times}}.
$$
Hopefully, there will be no confusion with the exponential $a^b$.
\begin{theo}
\label{chilenos2017}
 Let $\P^{\frac{n}{2}}, \check\P^{\frac{n}{2}}$ be two linear algebraic cycles in a smooth hypersurface of dimension $n$ and degree $d$ and  with the 
 intersection $\Pn m$. We have 
 \begin{equation}
 \label{6dec2016}
\intdim^d_n(m):=
  \codim(V_{\P^{\frac{n}{2}}}\cap V_{\check\P^{\frac{n}{2}}})=
  2\codnum_{1^{\frac{n}{2}+1},(d-1)^{\frac{n}{2}+1}}-\codnum_{1^{n-m+1}, (d-1)^{m+1}}.
 \end{equation}
In particular,  if $\P^{\frac{n}{2}}$ does not intersect $\check\P^{\frac{n}{2}}$ then
$V_{\P^{\frac{n}{2}}}$ intersects $V_{\check\P^{\frac{n}{2}}}$  transversely. 
\end{theo}
The proof is a simple application of Koszul complex and can be found in Section 17.9 of \cite{ho13}. 

We are now going to analyze the number $\mydim^d_n(m)$ for $m=\frac{n}{2},\frac{n}{2}-1,\cdots$. 
Let $\P^{\frac{n}{2}}$ and $\check\P^{\frac{n}{2}}$ be as in Introduction. 
Let us first consider the case $m=\frac{n}{2}$.
For the proof of Theorem \ref{main1} we have verified the first equality in
\begin{eqnarray*}
 \mydim^d_n(\frac{n}{2})&=&\intdim^d_n(\frac{n}{2})=
 \codnum_{1^{\frac{n}{2}},(d-1)^{\frac{n}{2}}}.
 \end{eqnarray*}
 (the second equality follows from Theorem \ref{chilenos2017}).  
 One of the by-products of the proof  is that   
$V_{\Pn {\frac{n}{2}}}$ as an analytic scheme  is smooth and reduced.  
For $m=\frac{n}{2}-1$,  we have  
\begin{eqnarray*}
\mydim^d_n(\frac{n}{2}-1) &=& \codnum_{1^{\frac{n}{2}},2,(d-1)^{\frac{n}{2}},d-2}\leq 
\intdim^d_n(\frac{n}{2}-1)=2\codnum_{1^{\frac{n}{2}+1},(d-1)^{\frac{n}{2}+1}}-\codnum_{1^{\frac{n}{2}+2}, (d-1)^{\frac{n}{2}}}.
\end{eqnarray*}
The first equality is conjectural and we can verify it for special cases of $n$ and $d$ by a computer,
see \cite{Roberto}, Section 5. In this case the algebraic cycle $\P^{\frac{n}{2}}+ \check\P^{\frac{n}{2}}$ can be deformed
into a complete intersection algebraic cycle of type $(1^{\frac{n}{2}},2)$, 
and so, the inequality is justified. Since the underlying complex variety of 
the Hodge locus $V_{[\P^{\frac{n}{2}}]+[\check\P^{\frac{n}{2}}]}$ contains $V_{\P^{\frac{n}{2}}}\cap V_{\check\P^{\frac{n}{2}}}$, 
Theorem \ref{main1.5} and Theorem \ref{chilenos2017} imply that the inequality 
\begin{equation}
\label{13mar2017}
\mydim^d_n(m)\leq  \intdim^d_n(m)
\end{equation}
holds for arbitrary $m$ between $-1$ and $\frac{n}{2}$.   
We conjecture that 
$$
\mydim^d_n(-1)=2\cdot \codnum_{1^{\frac{n}{2}+1},(d-1)^{\frac{n}{2}}+1}
$$
which is the value of $\intdim^d_n(m)$ (note that $\codnum_{1^{n+1}}=0$).  
This is the same as to say that:
\begin{conj}\rm
\label{emilio2017}
Let $\P^{\frac{n}{2}}, \check\P^{\frac{n}{2}}$ be two linear algebraic cycles in the Fermat variety and with no common point.
The only deformations of $\P^{\frac{n}{2}}+\check\P^{\frac{n}{2}}$ as an algebraic or Hodge cycle is again a sum of two linear cycles.  
\end{conj}
Particular cases of this conjecture has been announced in Theorem \ref{27.03.17} (those with $m=-1$).  
It might happen that in \eqref{13mar2017} we have a strict inequality, see for instance Table \ref{12mar2017}. 
\begin{conj}
\label{21m2017}
 For $n\geq 6$ we have 
 \begin{equation}
\label{13.03.2017}
\mydim^d_n(\frac{n}{2}-2)<\intdim^d_n(\frac{n}{2}-2).
\end{equation}
\end{conj}
Our favorite examples for verifying Conjecture \ref{21m2017}  are cubic Fermat varieties, that is $d=3$. 
For  $n\geq 4$ we have the following range:
 \begin{equation}
 \label{happyandsad-bonn}
\binom{\frac{n}{2}+1}{3}\leq  {\rm rank}([\pn_{i+j}]) \leq  \binom{n+2}{{\rm min }\{3, \frac{n}{2}-2\}}
 \end{equation}
 \href{http://w3.impa.br/~hossein/WikiHossein/files/Singular%20Codes/2017-03-IntersectionOfTwoLinearCycles.txt}
 {and in Table \ref{12mar2017} we have computed $\mydim^3_n(m)$ for $4\leq n\leq 10$ and $-1\leq m \leq \frac{n}{2}$.}
 The following table is the main evidence for Conjecture \ref{21m2017}. 
 \begin{table}[H]
\begin{tabular}{|c||c|c|c|c|c|c|c|c|}
\hline
$n\backslash (\frac{n}{2}-m)$    
        & $0$    &  $1$   &  $2$   & $3$   & $4$  & $5$  & $6$  &  $7$ \\  \hline \hline 
 $4$    & $(1,1)$  & $(1,2)$  & $(1,2)$  & $(1,2)$  &      &      & &   \\  \hline  
 $6$    & $(4,4)$  & $(4,7)$   & $(6,8)$   & $(7,8)$ & $(8,8)$&      & &   \\ \hline  
  $8$   & $(10,10)$& $(10,16)$  & $(16,19)$ & $(19,20)$ &$(20,20)$&$(20,20)$&  &     \\  \hline 
  $10$  & $(20,20)$ & $(20, 30)$ &  $(32,36)$      &   $(38,39)$    &   $(40,40)$  &   $(40,40)$   & $(40,40)$ &    \\  \hline 
  $12$  & $(35, 35)$& $(35,50)$  &  $(55,60)$      &   $(65,66)$    &   $(69,69)$  &   $(70,70)$   &  $(70,70)$ & $(70,70)$ \\ \hline
\end{tabular}
\caption{The numbers $(\mydim^3_n(m),\intdim^3_n(m))$.} 
\label{12mar2017}
\end{table}
We were also able to compute the five-tuples $(n,d,m| \mydim^d_n(m),\intdim^d_n(m))$ in the list below: 
\begin{eqnarray*}
 & & (4,4,0|11,12), \ (4,4, -1|12,12), \\
 & & (4,5,0|24,24),\  (4,5,-1|24, 24), \\
 & & (4,6,0|38, 38), \ (4,6,-1|38, 38),  \\
 & & (6,4,1|36,37),\  (6,4,0|38, 38),\   (6,4,-1| 38, 38). \\
\end{eqnarray*} 
We were not able to compute more data such as $?$ in $(4,7,0|?, 54)$. For $n=2$ and $4\leq d\leq 14$ 
we were also able to check Conjecture \ref{emilio2017}. Note that for the quartic Fermat fourfold we have
the range $6\leq \rank([\pn_{i+j}])\leq 21$ and $\T_{1,1,2}$ has codimension $8$.

\begin{proof}[Proof of Theorem \ref{27.03.17} for $\cf=\check\cf=1$] 
This is just the outcome of above computations in which $\mydim^d_n(m)=\intdim^d_n(m)$.
The full proof will be given after Theorem \ref{01.j.2017}. For $\cf=\check\cf=1$ we have 
Theorem \ref{27.03.17} for $(n,d,m)$ in
\begin{equation}
\label{9.9.2017}
(12,3,-1), (12,3,0),(12,3,1), (12,3,2),
\end{equation}
however, we were not able to verify Theorem \ref{01.j.2017}  in these cases.
\footnote{For the computations of $\mydim^d_n(m)$ and $\intdim^d_n(m)$ we have used the procedures 
{\tt SumTwoLinearCycle} and {\tt Codim}, respectively.}
\end{proof}

For the convenience of the reader we have also computed the table of Hodge numbers for cubic Fermat varieties. 
 Note that for $d=3,n=4$ the Hodge conjecture is well-known, see \cite{zu77}. 
\begin{table}[H]
\label{12mar2017-2}
\begin{tabular}{|c|c|c|c|}
\hline
$n$     & $\bn{3+n+1}{3}-(n+2)^2$  & $\binom{\frac{n}{2}+1}{3}$, $\binom{n+2}{{\rm min }\{3, \frac{n}{2}-2\}}$  
        &    Hodge numbers \\  \hline \hline 
 $4$    & $20$ &  $1,1$  & $0,1,21,1,0$\\  \hline  
 $6$    & $56$ &  $4,8$  & $0,0,8,71,8,0,0$ \\ \hline  
  $8$   & $120$ & $10,45$  &  $0, 0, 0, 45, 253, 45, 0, 0, 0$ \\  \hline 
  $10$  & $220$ & $20, 220$  & $0, 0, 0, 1, 220, 925, 220, 1, 0, 0, 0$ \\  \hline
  $12$  & $364$ & $35,1001$   &   $0,0,0,0, 14,1001, 3432, 1001, 14,0,0,0,0$   \\ \hline 
\end{tabular}
\caption{Hodge numbers} 
\end{table}

\begin{theo}
\label{01.j.2017}
For all pairs $(n,d)$ in Theorem \ref{27.03.17} with arbitrary $-1\leq m\leq \frac{n}{2}$ and all  
$x\in \Q$ with $x\not=0$, we have
\begin{equation}
\label{1july2017}
 \rank([\pn_{i+j}(\P^{\frac{n}{2}}+ x\check\P^{\frac{n}{2}}  )])=
  \rank([\pn_{i+j}(\P^{\frac{n}{2}}+ \check\P^{\frac{n}{2}}  )])
\end{equation}
and so this number does only depend on $(n,d,m)$ and not on $x$. 
\end{theo}
\begin{proof}
Let $a:=\mydim_n^d(m)$ be the number in the right hand side of \eqref{1july2017} and
let $A(x):=[\pn_{i+j}(\P^{\frac{n}{2}}+ x\check\P^{\frac{n}{2}}  )]$.
Except for a finite number of $x\in\Q$, we have $\rank(A(x))\geq a$ and in order to prove the equality, 
it is enough to check it for $a+2$ distinct
values of $x$. This is because if $\rank(A(x))>a$ then we have a $(a+1)\times (a+1)$ minor
of $A(x)$ whose determinant is not zero. This is a polynomial of degree at most $a+1$ in $x$,  and it has
$(a+2)$ roots which leads to a contradiction. This argument
implies that except for a finite number   of values for $x$ we have $\rank(A(x))=a$. 
These are the roots of $\det(B(x))=0$, where $B$ is any  $a\times a$ minor of $A(x)$ such
that $P(x):=\det(B(x))$ is not identically zero.
We find such a minor and  compute
$\rank(A(x))$ for all rational roots of $P(x)$ and prove that this is $a$ except for $x=0$.
\footnote{See {\tt GoodMinor} and {\tt ConstantRank}.}
\href{http://w3.impa.br/~hossein/WikiHossein/files/Singular%20Codes/2017-08-ConstantRankData}
{It seems interesting that only for $(n,d,m)= (6,3,1)$, $(6,3,0)$, $(8,3,2)$,$(8,3,1)$,$(6,4,1)$,$(10,3,3)$, $(10,3,2)$
we find  a rational root of $P(x)$, and in all these cases it is $x=-1$.}
This seems to have some relation 
with Conjecture \ref{CanISayFinally?} for $(\cf,\check\cf)=(1,-1)$.
\end{proof}
\begin{proof}[Proof of Theorem \ref{27.03.17}:]
For all the cases
in Theorem \ref{27.03.17} 
$$
\rank\left(\left [\cf\pn_{i+j}(\P ^{\frac{n}{2}})+\check\cf\check\pn_{i+j}(\check \P ^{\frac{n}{2}})
\right]\right)=
\rank\left(\left [\pn_{i+j}(\P ^{\frac{n}{2}})+\check\pn_{i+j}(\check \P ^{\frac{n}{2}})
\right]\right)
=
\intdim_n^d(m),
$$
where for the first equality we have used Theorem \ref{01.j.2017}. 
We know that $V_{\P^{\frac{n}{2}}}\cap V_{\check\P^{\frac{n}{2}}}$ is the subset of
the analytic variety underlying 
$V_{\cf[\P^{\frac{n}{2}}]+\check\cf[\check\P^{\frac{n}{2}}]}$ and its codimension is 
$\intdim_n^d(m)$. This proves the theorem. 
\end{proof}

\section{Smooth and reduced Hodge loci}
\label{tanhai2017}
Based on the computation in \S\ref{eid1396}, we have formulated Conjecture \ref{21m2017}, and we further claim that: 
\begin{conj}
\label{Duisburg2016}
Let $\P^{\frac{n}{2}}$ and $\check\P^{\frac{n}{2}}$ be two linear cycles in the Fermat variety $X^d_n$ with
$d\geq 2+\frac{4}{n}$ and  $\P^{\frac{n}{2}}\cap \check\P^{\frac{n}{2}}=\P^{m}$ with    
$-1\leq m\leq \frac{n}{2}-1$  and $\mydim^d_n(m)<  \intdim^d_n(m)$.
There is a finite number of coprime non-zero integers $\cf,\check\cf$ such that 
the analytic scheme 
$V_{\cf[\P^{\frac{n}{2}}]+\check \cf[\check\P^{\frac{n}{2}}]}$ 
is smooth and reduced. 
\end{conj}
If $\check\cf=0$ and  $\cf=1$ then we have the Hodge locus $V_{[\P^{\frac{n}{2}}]}$ which is
smooth and reduced by Theorem \ref{main2}.
Conjecture \ref{Duisburg2016} is true in the following case:  
$m=\frac{n}{2}-1$ and
$$
(n,d)=(2,d),\ 4\leq d\leq 15 \ \ \ (4,d), \  d=3,5,6,\ \ \ (6,d),\ \  d=3,4.  
$$
In this case, the Hodge locus $V_{[\P^{\frac{n}{2}}]+[\check\P^{\frac{n}{2}}]}$ is smooth and reduced at $0$ and it  
parameterizes hypersurfaces with a 
complete intersection of type
$(1^{\frac{n}{2}}, 2)$, see the comments before Theorem \ref{27.03.17}. The proof can be found in 
\cite{Roberto}. The analytic scheme $V_{\cf[\P^{\frac{n}{2}}]+\check\cf [\check\P^{\frac{n}{2}}]}$ is non-reduced or singular at $0$ in 
the cases covered in Theorem \ref{dream2017}. Other evidences to Conjecture \ref{Duisburg2016} are listed in Theorem 
\ref{WillCome2017} and Theorem \ref{dream2017}.

Assuming the Hodge conjecture, 
the points of the Hodge locus $V_{\cf[\P^{\frac{n}{2}}]+\check\cf [\check\P^{\frac{n}{2}}]}$ parametrizes
hypersurfaces with certain algebraic cycles. We do not have any idea how such algebraic cycles look like.  
In order to  verify Conjecture \ref{Duisburg2016} without constructing algebraic cycles, 
we have to analyze the the generators $\int_{\delta_t}\omega_i$ of 
the defining ideal  of the Hodge locus in \eqref{10maio16}. These are integrals depending on the parameter $t\in\T$ and their 
linear part is gathered in the matrix $[\pn_{i+j}]$. 
If Conjecture  \ref{Duisburg2016}  is true in these cases  then we have
discovered a new Hodge locus, different from $V_{[\P^{\frac{n}{2}}]}$, $V_{[\check\P^{\frac{n}{2}}]}$ and their intersection. 
The whole discussion of \S\ref{aima2017}
has the goal to provide tools to analyze Conjecture \ref{Duisburg2016}.

\section{The creation of a formula}
\label{aima2017}
In this section we compute the Taylor series of 
the integration of differential forms 
over monodromies of the algebraic cycle $\licy{\frac{n}{2}}_{a,b}$ inside the Fermat variety. 
Let us consider the hypersurface $X_t$ in the projective space 
$\Pn {n+1}$ given by the homogeneous polynomial:
\begin{equation}
\label{15dec2016}
f_t:=x_0^{d}+x_1^{d}+\cdots+x_{n+1}^d-\sum_{\alpha}t_\alpha x^\alpha=0,\ \ 
\end{equation}
$$
t=(t_\alpha)_{\alpha\in I}\in(\T,0), 
$$
where $\alpha$ runs through a finite subset $I$ of 
$\N_0^{n+2}$ with $\sum_{i=0}^{n+1} \alpha_i=d$. 
In practice, we will take the set $I$ of all such $\alpha$ with the additional 
constrain $0\leq \alpha_i\leq d-2$. For a rational number $r$ 
let $[r]$ be the integer part of $r$, that is $[r]\leq r<[r]+1$, and  $\{r\}:=r-[r]$. Let
also $(x)_y:=x(x+1)(x+2)\cdots(x+y-1),\ (x)_0:=1$ be the Pochhammer symbol. For $\beta\in \N_0^{n+2}$,  
$\bar\beta\in \N_0^{n+2}$ is defined by the rules:
$$
0\leq \bar\beta_i \leq d-1,\ \  \beta_i\equiv_{d}\bar \beta_i.
$$
\begin{theo}
\label{InLabelNadasht?}
Let $\delta_{t}\in H_n(X_t,\Z),\ t\in(\T,0)$ be the monodromy (parallel transport) of the cycle 
$\delta_0:=[\licy{\frac{n}{2}}_{a,b}]\in H_n(X_0,\Z)$ along a path which 
connects $0$ to $t$. 
For a monomial $x^\beta=x_0^{\beta_0} x_1^{\beta_1}x_2^{\beta_2}\cdots x_{n+1}^{\beta_{n+1}}$ with 
$k:=\sum_{i=0}^{n+1}\frac{\beta_i+1}{d}\in\N$  
we have 
 \begin{equation}
 \label{15.12.16}
 \frac{C }{  (2\pi \sqrt{-1})^{\frac{n}{2}}}
 \mathlarger{\mathlarger{\int}}_{\delta_t}\Resi\left(\frac{x^\beta\Omega}{f^{k}_t}\right)=
\mathlarger{\mathlarger{\mathlarger{\sum}}}_{a: I{}\to \N_0}
\left(
\frac{1}{ a! } D_{\beta+ a^*}\cdot   e^{ \pi \sqrt{-1}\cdot  {E}_{\beta+ a^*}} \right) \cdot  t^a,
\end{equation}
where  the sum runs through all $\#I{}$-tuples $a=(a_\alpha,\ \ \alpha\in I{})$
of non-negative integers such that for $\check\beta:=\beta+a^*$ we have 
\begin{equation}
\label{TheLastMistake2017}
\left\{\frac{ \check\beta_{b_{2e}}+1}{d} \right\}+   \left\{\frac{ \check\beta_{b_{2e+1}}+1}{d} \right\}=1,\ \ \ 
\forall  e=0,...,\frac{n}{2}, 
\end{equation}
and 
\begin{eqnarray*}
C &:=& 
{\rm sign}(b) \cdot (-1 )^{\frac{n}{2}} \cdot   d^{\frac{n}{2}+1}  \cdot (k-1)!,\\
t^a:&=&\prod_{\alpha\in I{}}t_\alpha^{a_\alpha}, \ \   \ \ \ \ \ \ \ |a| := \sum_{\alpha\in I{}}a_{\alpha},\\
a!&:=&\prod_{\alpha\in I{}}a_\alpha!,  \ \   \ \ \ \ \ \ \  a^* := \sum_{\alpha}a_\alpha\cdot \alpha, \\ 
D_{\check\beta} &:=& 
\mathlarger{\prod}_{i=0}^{n+1}\left( \left\{\frac{\check\beta_i+1}{d}\right\}\right)_{ \left[\frac{\check\beta_i+1}{d}\right ]},
 \\
 E_{\check\beta} &:= & 
		   \sum_{e=0}^\frac{n}{2} \left\{\frac{\check\beta_{b_{2e}}+1}{d}\right\}\cdot (1+2a_{2e+1}) 
\end{eqnarray*}
\end{theo}
This theorem is the outcome of many computations in \cite{ho13}. Its proof is obtained after a careful analysis of the Gauss-Manin
connection of the full family of hypersurfaces around the Fermat point $0\in\T$. For thus see Sections 13.9, 13.10, 17.11 of this book. 
In the next paragraph we are going to explain how to use Theorem \ref{InLabelNadasht?} and give evidences for Conjecture \ref{Duisburg2016}.

Recall the definition of the Hodge locus as an scheme in \eqref{10maio16}. 
Let $f_1,f_2,\cdots,f_a\in \O_{\T,0}$ be the integrals such that
$f_1=f_2=\dots=f_a=0$ is the underlying analytic variety of the Hodge locus $V_{\delta_0}$. 
We take $f_1,f_2,\ldots, f_k,\ \ k\leq a$ 
such that the linear part of $f_1,f_2,\ldots, f_k$ form a basis of the vector space generated 
by the linear part of all $f_1,f_2,\ldots, f_a$. By Griffiths transversality those of $f_i$ which come from
$F^{\frac{n}{2}+2}H^n_\dR(X_0)$ have zero linear part and so only $F^{\frac{n}{2}+1}/ F^{\frac{n}{2}+2}$ part of the
cohomology contribute to the mentioned vector space, see for instance \cite{ho13} Section 16.5.   

The Hodge locus $V_{\delta_0}$ is smooth and reduced if and only if the two ideals 
$\langle f_1,f_2,\ldots,f_k\rangle$ and $\langle f_1,f_2,\ldots,f_a\rangle$ in $\O_{\T,0}$ are the same. 
For this we have to check 
\begin{equation}
 f\in \langle f_1,f_2,\ldots,f_k,\rangle \ \hbox{ for } f=f_i,\ \ i=k+1,\ldots, a,   
\end{equation}
or equivalently 
\begin{equation}
\label{gavrilov2017}
f=\sum_{i=1}^k f_ig_i,\ \ g_i\in \O_{\T,0}. 
\end{equation}
Let $f=\sum_{j=1}^\infty f_i,\ \ f_i=\sum_{j=1}^\infty f_{i,j}, g_i=\sum_{j=0}^\infty g_{i,j}$ 
be the homogeneous decomposition of $f$, $f_i$ and $g_i$, respectively.
The identity \eqref{gavrilov2017} reduces to infinite number of polynomial identities:
\begin{eqnarray}
\label{15june2017}
f_1 &=& \sum_{i=1}^k f_{i,1}g_{i,0}, \\ \nonumber
f_2 &=& \sum_{i=1}^k f_{i,2}g_{i,0}+ \sum_{i=1}^k f_{i,1}g_{i,1}, \\ \nonumber
\vdots & & \vdots \\ \nonumber
f_j &=& \sum_{i=1}^k f_{i,j}g_{i,0}+ \sum_{i=1}^k f_{i,j-1}g_{i,1}+\cdots+  \sum_{i=1}^k f_{i,1}g_{i,j-1}. 
\end{eqnarray}
\begin{defi}\rm
 For a Hodge locus $V_{\delta_0}$ as in \eqref{10maio16} and $N\in\N$ we say that it is $N$-smooth if the first 
 $N$ equations in \eqref{15june2017} holds for all $f=f_i, i=k+1,k+2,\cdots,a$. 
 In other words \eqref{gavrilov2017} holds up to monomials of degree $\geq N+1$. 
\end{defi}
By definition a Hodge locus $V_{\delta_0}$ is $1$-smooth.
Theorem \ref{WillCome2017} and Theorem \ref{dream2017}, and in particular their computational proof, 
must be considered our strongest evidence to Conjecture \ref{Duisburg2016}. 
\begin{proof}[Proof of Theorem \ref{WillCome2017} and Theorem \ref{dream2017}  ]
\href
{http://w3.impa.br/~hossein/WikiHossein/files/Singular%20Codes/2017-09-ReducedSmoothOutputFinal}
{The proof is done using a computer implementation of the Taylor series \eqref{15.12.16}.}
 \footnote{See {\tt SmoothReduced} and {\tt TaylorSeries}.}
 In order to be sure that this Taylor series and its computer implementation are mistake-free we 
 have also checked many $N$-smoothness property which are already proved in Theorem \ref{27.03.17}. 
 In Theorem \ref{dream2017} Item 1 we have proved that the corresponding 
 Hodge locus is not $2$-smooth except
 in the following case which we highlight it. 
 Let $\P^1$ and $\check\P^1$ be two lines in the Fermat quintic surface intersecting 
 each other in a point. The Hodge locus $V_{\cf\P^1+\check\cf\check\P^1}$  for all $\cf,\check\cf\in \Z$ is $2$-smooth. Moreover it
 is not $3$-smooth for  $0<|\cf|<|\check\cf|\leq 10$. In Theorem \ref{dream2017} Item 2 (resp. 3) 
 we have proved that  the corresponding Hodge locus is not $3$-smooth (resp. $4$-smooth). 
 
%
%
\end{proof}
The property of being $N$-smooth for larger $N$'s is out of the 
capacity of my computer codes, see \S\ref{10.08.2017} for some comments.

%
%

\section{Uniqueness of components of the Hodge locus}
\def\Ho{{\sf Ho}}
\label{25.03.2017}
A Hodge cycle $\delta\in H_n(X^d_n,\Z)$ is uniquely determined by its periods $\pn_i(\delta)$.
This data gives the Poincar\'e dual of $\delta$ in cohomology, and hence, the 
classical Hodge class in the literature. Let $\Ho_n^d$ be  the $\Z$-module of period vectors $\pn$ of Hodge cycles.
We will also use its projectivization $\P \Ho^d_n$ (two elements $\pn$ and $\check\pn$ in the $\Z$-module are the same if 
there are non-zero integers $a$ and  $\check a$ such that $a\pn=\check a\check\pn$). 
This $\Z$-module can be described in an elementary linear algebra context without referring to advanced topics, such
as homology and algebraic de Rham cohomology, see Chapter 16 of \cite{ho13}. 
Therefore, the conjectures of the present section can be understood by 
any undergraduate mathematics student!
If 
either $d$ is a prime number or $d=4$ or $d$ is relatively prime with $(n + 1)!$ then  
we may redefine  $\Ho_d^n$ the $\Z$-modules generated by $\pn^{a,b}$, where  
$$
\pn^{a,b}_i:=
\left\{
	\begin{array}{ll}
		  
		   \zeta_{2d}^{\sum_{e=0}^\frac{n}{2} (i_{b_{2e}}+1)\cdot (1+2a_{2e+1}) }  & \mbox{if } \ \ \  
		   \ i_{b_{2e-2}}+i_{b_{2e-1}}=d-2, \ \ \ \forall e=1,...,\frac{n}{2}+1, \\
		0 & \mbox{otherwise. } 
	\end{array}
\right.
$$
and $a$ and $b$ are as in \eqref{6apr2017}. 
By  Theorem \ref{14out2016} and  Theorem \ref{24.07.2016} this will be 
a sub $\Z$-module of the $\Ho^d_n$ defined earlier. This will not modify our discussion below. 
\footnote{The list of $\pn^{a,b}$'s is implemented in the procedure {\tt ListPeriodLinearCycle}.} 
Recall the matrix
$[\pn_{i+j}]$ in Definition \ref{05/04/2017} and  the map $\P\Ho_n^d\to \N, \pn \mapsto \rank([\pn_{i+j}])$. 
If $\pn\not=0$ then
 \begin{equation}
 \label{happyandsad}
\binom{\frac{n}{2}+d}{d}-(\frac{n}{2}+1)^2 \leq {\rank}([\pn_{i+j}])\leq 
\left\{
\begin{array}{ll}
\binom{\frac{n}{2}d-1}{n+1} & \hbox{ if } d<\frac{2(n+1)}{n-2},\\
 \binom{d+n}{n+1}-(n+2) &     \hbox{ if } \ d=\frac{2(n+1)}{n-2},\\
\binom{d+n+1}{n+1}-(n+2)^2 &     \hbox{ if }  \ d > \frac{2(n+1)}{n-2}.
\end{array}\right.
 \end{equation}
 Before stating our main conjecture in this section, let us state a simpler one. 
 \begin{conj}
 \label{1apr2017}
  Let $n\geq 2$ be an even number and $d\geq 3$ an integer with $(n,d)\not=(2,4),(4,3)$. 
  Let also  $\pn\in \Ho^d_n$ such that 
  \begin{equation}
  \label{29.03.2018}
   \rank([\pn_{i+j}])=\binom{\frac{n}{2}+d}{d}-(\frac{n}{2}+1)^2.
  \end{equation}
Then $\pn$, up to multiplication by a rational number,  is necessarily of the form $\pn^{a,b}$. 
 \end{conj}
One can also formulate a similar conjecture for the next admissible rank.
For $n=2$ Voisin's result in \cite{voisin1988} tells us that this must be 
$2d-7:=\codim(\T_{1,2})$. For further discussion on this topic see \cite{ho13} Chapter 19.
It might happen that in Conjecture \ref{1apr2017} one must exclude more examples of $(n,d)$.
Note that for $(n,d)=(2,4),(4,3)$ both sides of \eqref{29.03.2017} are equal to one for all non-zero $\pn$. 


We need to write
down in an elementary language when the linear cycles $\P^\frac{n}{2}_i$ and $\check \P^ \frac{n}{2}_j$ underlying
two period vectors $\pn^i,\ i=(a,b)$ and $\check\pn^j,\ j=(\check a,\check b)$, respectively, have the intersection $\P^m$. This is as follows:
A bicycle attached to the permutations $b$ and $\check b$ is a sequence $(c_1c_2\ldots c_r)$ with 
$c_i\in\{0,1,2,\ldots,n+1\}$ and such
that if we define $c_{r+1}=c_1$ then for $1\leq i\leq r$ odd (resp. even) there is an even number $k$ with $0\leq k\leq n+1$  such that 
$\{c_i,c_{i+1}\}=\{b_{k},b_{k+1}\}$  (resp. $\{c_i,c_{i+1}\}=\{\check b_{k},\check b_{k+1}\})$) and  
there is no repetition
among $c_i$'s. By definition there is a sequence of even numbers $k_1,k_2,\cdots$ such that 
$$
\{c_1,c_2\}=\{b_{k_1},b_{k_1+1}\}, \ \ \{c_2,c_3\}=\{\check b_{k_2},\check b_{k_2+1}\},\ \  \{c_3,c_4\}=\{b_{k_3},b_{k_3+1}\},\ldots.
$$
Bicycles are defined up to twice shifting $c_i$'s, that is, $(c_1c_2c_3\cdots c_r)=(c_3\cdots c_rc_1c_2)$ etc., and the involution 
$(c_1c_2c_3\cdots c_{r-1}c_r)=(c_rc_{r-1}\cdots c_3c_2c_1)$. 
For example, for the permutations 
$$
b=(0,1,2,3,4,5), \ \check b=(1,0,5,3,4,2)
$$
we have in total two bicycles $(01), (2354)$. Note that bicycles give us in a natural way 
a partition of $\{0,1,\ldots,n+1\}$. 
For such a bicycle we define its conductor 
to be the sum over $k$, as before, of the following elements:
if $c_i=b_k$ and $c_{i+1}=b_{k+1}$ (resp. $c_i=\check b_k$ and $c_{i+1}=\check b_{k+1}$) then the element 
$1+2a_{k+1}$ (resp. $1+2\check a_{k+1}$), 
and if  if $c_i=b_{k+1}$ and $c_{i+1}=b_{k}$ (resp. $c_i=\check b_{k+1}$ and $c_{i+1}=\check b_{k}$)    
then $-1-2a_{k+1}$ (resp. $-1-2\check a_{k+1}$). Because of the involution, the conductor 
is defined up to sign. In our example, the
conductor of $(01)$ and $(2354)$ are respectively given by 
$$
1+2a_1+1+2\check a_1,\ \  
1+2a_3-1-2\check a_3-1-2a_5+1+2\check a_5.
$$ 
A bicycle is called
new if $2d$ divides its conductor, and is called old  otherwise.
Let $m_{ij}$ be the number
of new bicycles attached to $(i,j)$ minus one.
\begin{conj}
\label{TrainFreiburgBonn}
 Let $n\geq 4$ and $d>\frac{2(n+1)}{n-2}$. If for some  $\pn\in\Ho^d_n,\ \ \pn\not=0$ we have  
 $$
 \rank[\pn_{i+j}]\leq \mydim^d_n(\frac{n}{2}-2),
 $$
 then $\pn$ after multiplication with a natural number is in the set
 \begin{enumerate}
  \item 
 $\Z\pn^{a,b}$ and so $\rank([\pn_{i+j}])=
 \bn{\frac{n}{2}+d}{d}-(\frac{n}{2}+1)^2$.
  \item
  $\Z\pn^{a,b}+\Z\pn^{\check a,\check b}$ with $m_{a,b,\check a,\check b}=\frac{n}{2}-1$ and
  so $\rank([\pn_{i+j}])= \codnum_{1^{\frac{n}{2}},2,(d-1)^{\frac{n}{2}},d-2}$.
  \item
 $\Z\pn^{a,b}+\Z\pn^{\check a,\check b}$ with $m_{a,b,\check a,\check b}=\frac{n}{2}-2$ and
 so $\rank[\pn_{i+j}]= \mydim^d_n(\frac{n}{2}-2)$.
 \end{enumerate}  
\end{conj}
A complete analysis of Conjecture \ref{TrainFreiburgBonn} would require an intensive search for
the elements $\pn\in\Ho^d_n$ of low $\rank([\pn_{i+j}])$. 
It might be true for $n=2$ and large $d$'s, and 
this has to do with the Harris-Voisin
conjecture, see \cite{GMCD-NL}, and will be discussed somewhere else. Note that the numbers in items 1,2,3 of
Conjecture \ref{TrainFreiburgBonn} for $n=2$ are respectively $d-3,2d-7$ and $2d-6$ (for the last one see 
Conjecture \ref{emilio2017}). 
We just content ourselves with the following strategy  for confirming Conjecture \ref{TrainFreiburgBonn}.   
Let  $\pn^i,\ i=1,2,3$ be three distinct vectors of the 
form $\pn^{a,b}$. We claim that for $d>3$ we have  
\begin{equation}
\label{hanooztalash2017}
\rank([\pn_{i+j}])>\mydim^d_n(\frac{n}{2}-2), \hbox{ where }  \pn=\pn^1+\pn^2+\pn^3. 
\end{equation}
The number $\mydim^d_n(\frac{n}{2}-2)$ is  computed in \S\ref{eid1396} and
so we check in total $\bn{N}{3}$ inequalities \eqref{hanooztalash2017}, where $N$ is the
number of $\pn^{a,b}$'s in \eqref{28.06.2017}. 
This is too many computations and we have checked \eqref{hanooztalash2017}
for samples of $\pn^i$'s for $(n,d)=(4,6)$. In this way we have also observed that the lower bound 
for $d$ is necessary as \eqref{hanooztalash2017} is not true for our  favorite examples $(n,d)=(4,4),(6,3)$.
For $d=3$, the vector $\pn$ in \eqref{hanooztalash2017} can be zero. 
\footnote{For this computations we have used the  procedure {\tt SumThreeLinearCycle}.}

The final ingredient of Conjecture \ref{CanISayFinally?} is the following. In virtue of Theorem \ref{main1.5}, it compares the Zariski tangent
spaces of components of the Hodge locus passing through the Fermat point. 
\begin{conj}
\label{13jan2017}
Let $n\geq 6$ and $d\geq 3$. There is no inclusion between any two vector spaces of the form  
\begin{equation}
\label{27.06.17}
\ker([\cf\pn_{i+j}+\check \cf\check\pn_{i+j}])
\end{equation}
where $\pn$ and $\check \pn$ ranges in the set of all $\pn^{a,b}$ with $m_{a,b,\check a, \check b}=\frac{n}{2}-1,
\frac{n}{2}-2$, $\cf,\check\cf\in \Z$ coprime and $m_{a,b,\check a, \check b}=\frac{n}{2}, \cf=1,\check\cf=0$. 
\end{conj}
Let $\delta,\ \check\delta\in H_n(X^d_n,\Q)$ be two Hodge cycles with 
\begin{equation}
\label{Emanuel2017}
\ker([\pn_{i+j}(\delta)])\subset \ker([\pn_{i+j}(\check \delta)]),
 \end{equation}
that is, the Zariski tangent space of $V_{\delta}$ is contained in the Zariski tangent space of $V_{\check\delta}$.  The first
trivial example to this situation is when $\check\delta$ is a rational multiple of $[\plc]$ for which we have $\pn(\check\delta)=0$ and
$V_{\check\delta}=(\T,0)$. Let us assume that none of  $\check\delta$ and $\delta$ is a rational multiple of $[\plc]$.  
Next examples for this situation are in Theorem \ref{27.03.17}. In this theorem the Zariski tangent space 
of  the Hodge locus $V_{\cf[\P^{\frac{n}{2}}]+ \check\cf[\check \P^{\frac{n}{2}}]}=
V_{\P^{\frac{n}{2}}}\cap V_{\check \P^{\frac{n}{2}}}$, 
with  $\P^\frac{n}{2}\cap \check\P^\frac{n}{2}=\P^m$ and $\cf,\check\cf\in\Z,\ \cf\not=0,\ \check\cf\not=0$, 
at the Fermat point  does not depend on $\cf,\check\cf$. For larger $m$'s such as $\frac{n}{2}-1$, the Zariski tangent spaces of 
 $V_{\cf[\P^{\frac{n}{2}}]+ \check\cf[\check \P^{\frac{n}{2}}]} $ at the Fermat point form a pencil of 
 linear spaces and so there is no inclusion among its members. For $(n,d)=(2,4),(4,3)$, $V_{\delta}$'s are of codimension one, 
 smooth and reduced,  and so, any inclusion \eqref{Emanuel2017} will be 
 an equality and it implies that the period vectors of $\check\delta,\check\delta$ are the same. This implies that 
 $\delta=a\check\delta+b[\plc]$ for some $a,b\in\Q$, and so, $V_{\delta}=V_{\check\delta}$. 

We can verify Conjecture \ref{13jan2017} in the following way.
For simplicity we restrict ourselves to the pairs $(n,d)$ in Theorem \ref{27.03.17} and  $\cf=\check\cf=1$.  
Let us take two matrices $A$ and $B$ as inside kernel in \eqref{27.06.17}. 
Let also $A*B$ be the concatenation of $A$ and $B$ by putting the rows of $A$
and $B$ as the rows of $A*B$. Therefore, $A*B$ is a 
$(2\#I_{\frac{n}{2}d-n-2})\times (\# I_d)$ matrix. 
In order to prove that there is no inclusion between $\ker(A)$ and $\ker(B)$ it
is enough to prove that 
\begin{equation}
\label{pilar2017}
\rank(A*B)>\rank(A),\ \rank(B).
\end{equation}
The number of  verifications \eqref{pilar2017}  is approximately
$N^4$, where $N$ is the number of linear cycles given in \eqref{28.06.2017}. 
This is a huge number  even for small values of $n$ and $d$.\footnote{For this proof we have used 
{\tt DistinctHodgeLocus}. 
}
Note that the vector space in \eqref{27.06.17} for $(n,d,m)$'s in  Theorem \ref{27.03.17} 
is equal to the Zariski tangent space of $V_{\P^{\frac{n}{2}}}\cap V_{\check\P^{\frac{n}{2}}}$ at the Fermat point, and
hence it does not depend on $\cf$ and $\check\cf$. This is the main reason why we restrict ourselves to the cases in 
Conjecture \ref{13jan2017}.

\section{Semi-irreducible algebraic cycles}
\label{17agu2017}
Let $X$ be a smooth projective variety and $Z=\sum_{i=1}^r n_i Z_i,\ \ n_i\in\Z$ be an algebraic cycle in $X$, 
with $Z_i$ an irreducible subvariety of codimension $\frac{n}{2}$ in $X$.  
The following definition  is done using analytic deformations 
and it would not be hard to state it in the algebraic context. 
\begin{defi}\rm
 We say that $Z=\sum_{i=1}^r n_i Z_i,\ \ n_i\in\Z$ is semi-irreducible if there is a smooth analytic variety $\cal X$, an irreducible  subvariety ${\cal Z}\subset
 {\cal X}$ of codimension $\frac{n}{2}$ (possibly singular), 
 a holomorphic map  $f:{\cal X}\to (\C,0)$ such that
 \begin{enumerate}
  \item 
 $f$  is smooth and proper over $(\C,0)$ with $X$ as a fiber over $0$. Therefore, all the fibers  $X_t$ of $f$
 are $C^\infty$ isomorphic to $X$. 
  \item 
  The fiber $Z_t$ of $f|_{\cal Z}$ over $t\not=0$ is irreducible and $Z_0=\cup_{i=1}^r Z_i$.    
\item 
\label{17.08.2017}
The homological cycle $[Z]:=\sum_{i=1}^r n_i[Z_i]\in H_n(X,\Z)$ is the monodromy of $[Z_t]\in H_n(X_t,\Z)$.  
\end{enumerate}
\end{defi}
It is reasonable to expect that Item \ref{17.08.2017} is equivalent to a geometric phenomena, purely expressible in terms
of degeneration of algebraic varieties. For instance, one might expect that $n_i$ layers of the algebraic cycle $Z_t$ accumulate
on $Z_i$, and hence semi-irreducibility implies the positivity of $n_i$'s. Moreover,  for distinct 
$Z_i$ and $Z_j$, the intersection $Z_i\cap Z_j$ is of codimension one in both $Z_i$ and $Z_j$, 
because $Z_i$'s are irreducible and of 
codimension one in  $\cal Z$. In particular, the algebraic cycle 
$\cf\P^\frac{n}{2}+\check\cf\check\P^\frac{n}{2}$ with  $\P^\frac{n}{2}\cap \check\P^\frac{n}{2}=\P ^m,\ \ m\leq \frac{n}{2}-2$
in Conjecture \ref{CanISayFinally?} is not semi-irreducible.

A smooth hypersurface of degree $d$ and dimension $n$ has the Hodge numbers 
$h^{n,0}=h^{n-1,1}=\cdots=h^{\frac{n}{2}+2,\frac{n}{2}-2}=0, \ \ h^{\frac{n}{2}+1,\frac{n}{2}-1}=1$ 
if and only if $(n,d)=(2,4),(4,3)$. Recall that $\plc$ is the intersection of a linear $\P^{\frac{n}{2}+1}$ with $X^d_n$. 
The following theorem can be considered as a counterpart of Conjecture \ref{CanISayFinally?}.
\begin{theo}
\label{11Aug2017}
 Let $(n,d)=(2,4),(4,3)$ and let $Z$ be an algebraic cycle of dimension $\frac{n}{2}$ and with integer coefficients,  
 in a smooth hypersurface of dimension $n$ and degree $d$. If $[Z]\in H_n(X^d_n,\Q)$ is not a rational multiple of 
 $[\plc]$ then there is  a semi-irreducible algebraic cycle  $\check Z$ of dimension $\frac{n}{2}$  
 in $X^d_n$  such that 
 $a Z+b\check Z+c\plc$
 is homologous to zero for some $a,b,c\in\Z$ with $a,b \not=0$ .
\end{theo}
\begin{proof}
 The algebraic cycle $Z$ induces a homology class $\delta_0=[Z]\in H_n(X^d_n,\Z)$ and the 
 Hodge locus $V_{\delta_0}$ is given by the zero locus
of a single integral $f(t):=\int_{\delta_t}\omega_0$, where $\omega_0$ is given by \eqref{11.06.2017} for $i=(0,0,\cdots,0)$.
By our hypothesis on $Z$, $f$ is not identically zero and since $\delta_0=[Z]$ it vanishes at $0\in\T$.  
We show that $V_{\delta_0}$ is smooth and reduced, and for this it is enough to show that the linear part of $f$ is 
not  identically zero. 
This follows from $\nabla_{\frac{\partial}{\partial t_i}}\omega_0=\omega_i,\ \ i\in I_d$, $I_d=I_{(\frac{n}{2}+1)d-n-2}$ and the fact
that $\omega_0,\omega_i,\ i\in I_d$ form a basis of $F^1$ of $H^n_\dR(X)$. Here, $\nabla$ is the Gauss-Manin connection of the family
of hypersurfaces given by \eqref{15dec2016}. 
The Hodge conjecture in both cases is well-known.
In the first case  it is the Lefschetz $(1,1)$ theorem and in the second case it  is a result of Zucker in \cite{zu77}.
This implies that  $\delta_t=[Z_t]$, where  $Z_t:=\sum_{i=1}^r n_iZ_{i,t},\  Z_{i,t}\subset X_t,\ t\in V_{\delta_0}, 
\dim(Z_{i,t})=\frac{n}{2}, \   n_i\in\Q$ and for generic $t$, $Z_{i,t}$ is irreducible. 
Since $V_{\delta_0}\subset V_{[Z_{i,0}]}$, we conclude that $[Z_{i,0}]=a_i[Z]+b_i[\plc]$ for some $a_i,b_i\in\Q$. 
By our hypothesis on $Z$, one of $a_i$'s is not zero let us call it $a_1$. We get $[Z]=a_1^{-1}[Z_{1,0}]-b_1a_1^{-1}[\plc]$. 
\end{proof}
In Theorem \ref{11Aug2017} let us assume that $Z$ is a sum of linear cycles. 
It would be useful to see whether the algebraic cycle $\check Z$ is a sum of  linear cycles. One might start  
with the sum of two lines in the Fermat surface $X^4_2$ without any common points 
(the case $(n,d,m)=(2,4,-1)$). 

\begin{figure}
\begin{center}
\includegraphics[width=0.6\textwidth]{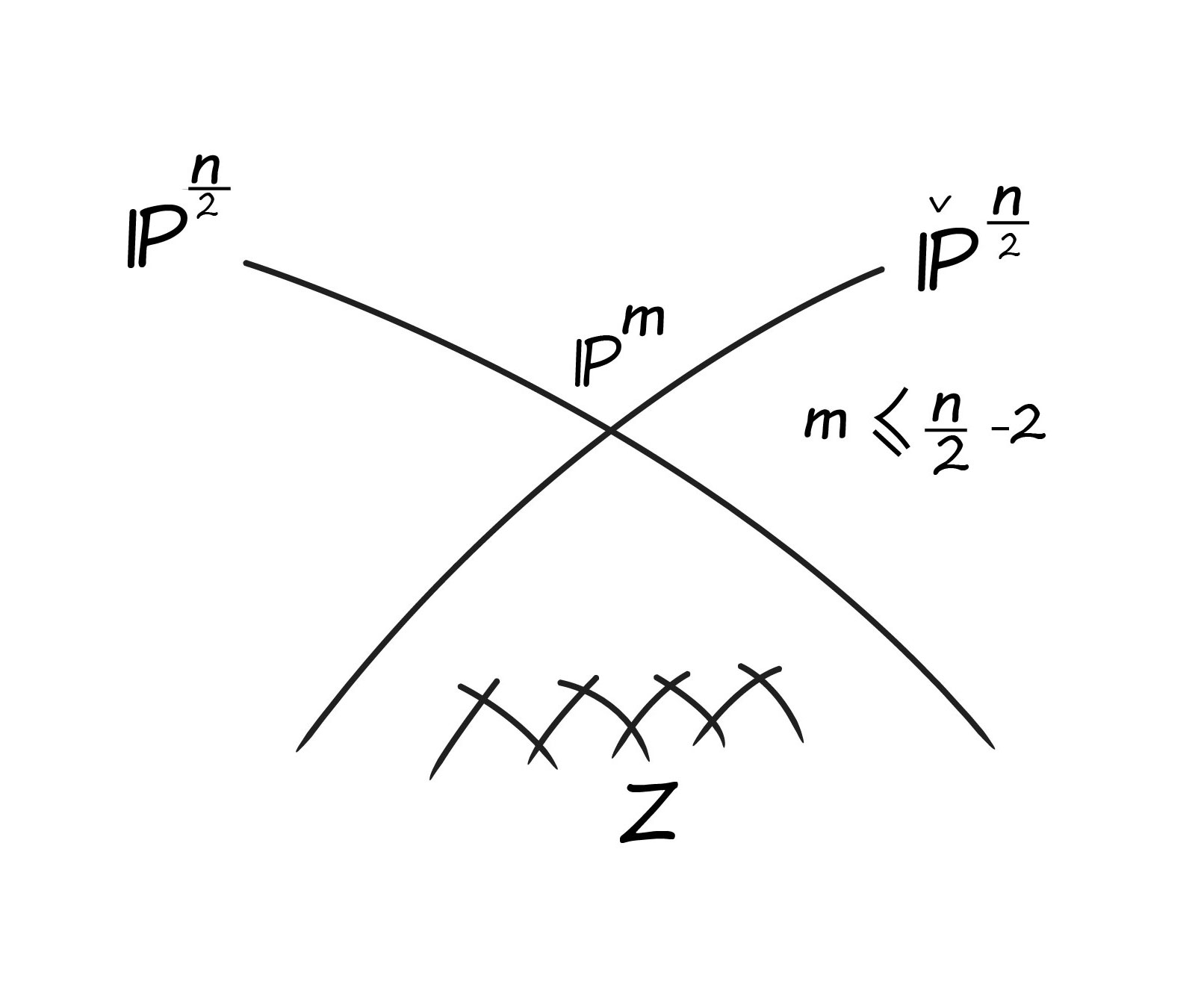}
\caption{Sum of linear cycles II}
\label{TwoP-2}
\end{center}
\end{figure}

\section{How to to deal with Conjecture \ref{CanISayFinally?}?}
\label{27.07.2017}
In this section we sketch a strategy to prove Conjecture \ref{CanISayFinally?} which 
follows the same guideline as of the proof of Theorem \ref{11Aug2017}.  
Let $\delta_0:=\cf[\P^{\frac{n}{2}}]+\check\cf[\check\P^{\frac{n}{2}}]\in H_n(X_n^d,\Z)$ with 
$\P^{\frac{n}{2}}\cap \check\P^{\frac{n}{2}}=\P^{m}, \ \ m=\frac{n}{2}-2$ and  $\mydim^d_n(m)<  \intdim^d_n(m)$.
Let also $\delta_t\in H_n(X_t,\Z),\ \ t\in (\T,0)$ be its monodromy to nearby fibers. 
Conjecture \ref{Duisburg2016}  implies that the intersection of $V_{\P^{\frac{n}{2}}}$ 
and $V_{\check\P^{\frac{n}{2}}}$ is a proper subset of  the underlying analytic variety of 
$V_{\delta_0}$. If the Hodge conjecture is
true then there is an algebraic family of algebraic cycles 
\begin{equation}
\label{sad2017}
Z_t:=\sum_{k=1}^r n_kZ_{k,t},\  Z_{k,t}\subset X_t,\ t\in V_{\delta_0},
\end{equation}
$$
\dim(Z_{k,t})=\frac{n}{2}, \   n_k\in\Z,\ 
$$
such that $Z_{k,t}$ is irreducible for generic $t$ and   
$Z_t$ is homologous to a non-zero integral  multiple of $\delta_t$, 
see Figure \ref{TwoP-2}.  
By Conjecture \ref{Duisburg2016} we know that $V_{\delta_0}$ is smooth and reduced, and so, 
we have the  inclusion of analytic schemes  
\begin{equation}
\label{12jan2017}
V_{\delta_0}\subset V_{[Z_{k,0}]},\ \ k=1,2,\ldots,r
\end{equation}
which implies that
\begin{equation}
\label{29.03.2017}
\ker[\pn_{i+j}(\delta_0)]\subset \ker[\pn_{i+j}(Z_{k,0})],\ \ \hbox{ and so }\  
\rank([\pn_{i+j}(Z_{k,0})])\leq \rank [\pn_{i+j}(\delta_0)].
\end{equation} 
In order to proceed, we consider the cases of Fermat varieties 
such that linear cycles generates the the space of Hodge cycles over rational numbers (these
are the cases in Theorem \ref{24.07.2016}), or we assume Conjecture \ref{TrainFreiburgBonn}
for $\Ho^d_n$ being the the lattice of periods of all Hodge cycles and not just linear cycles.  
We apply Conjecture  \ref{TrainFreiburgBonn} and we conclude that for some linear cycles 
$\P^{\frac{n}{2}}_k, \check \P^{\frac{n}{2}}_k$ in $X^d_n$ with 
$\P^{\frac{n}{2}}_k\cap\check \P^{\frac{n}{2}}_k=\P^{m_k}, \ \ m_k\geq \frac{n}{2}-2$ and 
$\cf_k,\check\cf_k,b_k,c_k\in \Z,\ c_k\not=0$ we have 
\begin{equation}
\label{18a2017}
\cf_k\P^{\frac{n}{2}}_k+\check\cf_k\check \P^{\frac{n}{2}}_k+b_k\plc+c_kZ_{k,0}\sim 0,
\end{equation}
where $\sim$ means homologous.  
The inclusion in \eqref{29.03.2017} and \eqref{18a2017} imply 
\begin{equation}
\label{mikham2017}
\ker[\pn_{i+j}(\cf\P^{\frac{n}{2}}+\check\cf\check \P^{\frac{n}{2}})]
\subset \ker[\pn_{i+j}(\cf_k\P^{\frac{n}{2}}_k+\check\cf_k\check \P^{\frac{n}{2}}_k )],\
\end{equation}
Now, Conjecture \ref{13jan2017}, 
\eqref{mikham2017} and the fact that $\cf$ and $\check\cf$ are coprime imply that for some non-zero integer 
$a$ we have 
$\cf_k\P^{\frac{n}{2}}_k+\check\cf_k\check \P^{\frac{n}{2}}_k=a(\cf\P^{\frac{n}{2}}+\check\cf 
\check \P^{\frac{n}{2}})$, as an equality of algebraic cycles, and hence   
$\{\P^{\frac{n}{2}}, \check \P^{\frac{n}{2}}\}= \{\P^{\frac{n}{2}}_k, \check \P^{\frac{n}{2}}_k\}  $. 
This means that in \eqref{sad2017} we can assume that 
$Z_t$ is irreducible  for generic $t$ and so we get 
\begin{equation}
\label{ObraPrimaDaSara}
 a(\cf\P^{\frac{n}{2}}+\check\cf\check \P^{\frac{n}{2}})+b\plc+cZ\sim 0,\ \ \ a,c\not=0, a,b,c\in\Z, 
 \end{equation}
where $Z=Z_0$. 
Taking the intersection of \eqref{ObraPrimaDaSara} with any third linear cycle $\tilde\P^\frac{n}{2}$ with 
$ \tilde\P^\frac{n}{2}\cdot \P^\frac{n}{2}=\tilde\P^\frac{n}{2}\cdot \check\P^\frac{n}{2}=0$ we get
$c\mid b$. Moreover,  taking the intersection of \eqref{ObraPrimaDaSara} with any third linear cycle $\tilde\P^\frac{n}{2}$ with 
$ \tilde\P^\frac{n}{2}\cdot \P^\frac{n}{2}=1$ and 
$\tilde\P^\frac{n}{2}\cdot \check\P^\frac{n}{2}=0$ we get $c\mid (a\cf+b)$. In a similar way, we have $c\mid (a\check\cf+b)$.
Since $\cf$ and $\check\cf$ are coprime we conclude that $c|a$ and, therefore, in \eqref{ObraPrimaDaSara}
we can assume that $c=-1$. 

One of the most important information about the algebraic cycle $Z\subset X^d_n$ is the data of its intersection numbers with other
algebraic cycles of the Fermat variety, and in particular all linear cycles. Recall that $\plc\cdot \plc=d$, 
for a linear cycle $\P^{\frac{n}{2}}\subset X^d_n$ we have $\plc\cdot \P^{\frac{n}{2}}=1$, and for two linear cycles 
$\P^{\frac{n}{2}}$ and $\check\P^{\frac{n}{2}}$ with $ \P^{\frac{n}{2}}\cap\check\P^{\frac{n}{2}}=\P^m$ we have
 \begin{equation}
 \label{13jan16-2}
\P^{\frac{n}{2}}\cdot \check\P^{\frac{n}{2}}=
\frac{1-(-d+1)^{m+1} }{d}
\end{equation}
This follows from the adjunction formula, see for instance \cite{ho13}, Section 17.6.  Using this we know $a$ and $b$: 
$$
b=Z\cdot \tilde\P^{\frac{n}{2}},\ \hbox{ where } 
\tilde\P^\frac{n}{2}\cdot \P^\frac{n}{2}=\tilde\P^\frac{n}{2}\cdot \check\P^\frac{n}{2}=0, 
$$
and 
$$
\deg(Z)=a\cdot(\cf+\check\cf)+b\cdot d. 
$$
In particular, if $(\cf,\check\cf)=(1,-1)$ then the degree $d$ of the Fermat variety divides the degree of $Z$. 
Another important information about the algebraic cycle $Z$ is a lower bound of the dimension of  the Hilbert scheme parameterizing 
deformations of the pair $(X^d_n,Z)$. 
One may look for the classification of the components of the Hilbert schemes of projective varieties in order to see whether such a $Z$
exists or not. For instance, we know that if $Z\subset \Pn {n+1}$ is an irreducible reduced projective variety of dimension 
$\frac{n}{2}$ and degree $2$ then it is necessarily a complete intersection of type $1^{\frac{n}{2}},2$, see \cite{EisenbudHarris1985}. 
One might look for  generalizations of this kind of results.

\section{Final comments}
\label{10.08.2017}

One of the main difficulties in generalizing our main theorems in Introduction for other cases 
is that the moduli of hypersurfaces of dimension $n$ and degree $d$ is of dimension 
\label{10ago2017}
$\#I_d=\bn{d+n+1}{n+1}-(n+2)^2$
which is two big even for small values of  $n$ and $d$. One has to prepare similar tables as 
in Table \ref{12mar2017} with smaller number of parameters and then start to analyze $N$-smoothness. 
For some suggestions see \cite{ho13} Exercises 15.13, 15.16, 15.17. The author has analyzed statements 
similar to Theorem \ref{WillCome2017} and Theorem \ref{dream2017} for hypersurfaces given by homogeneous polynomials of the form 
\begin{equation}
 f:=A(x_0,x_2,\ldots,x_n)+B(x_1,x_3,\ldots,x_{n+1}).
\end{equation}
The moduli of such hypersurfaces is of dimension $2\cdot \binom{d+\frac{n}{2}}{\frac{n}{2}}-2(\frac{n}{2}+1)^2$ and this makes 
the computations much faster. Here are some sample results mainly in direction of Theorem \ref{dream2017}. 
The Hodge locus $V_{\cf[\P^{\frac{n}{2}}]+\check\cf[\check\P^{\frac{n}{2}}] }$ for $\cf,\check\cf$ coprime non-zero integers and 
$|\cf|,|\check\cf|\leq 10$ is $7$-smooth  and $4$-smooth for  
$(n,d,m)=(6,3,1)$ and $(4,4,0)$, respectively. Therefore, it seems that we are in situations similar to Theorem \ref{27.03.17}.
For $(n,d,m)=(8,3,2),(10,3,3)$ the situation is similar to Theorem \ref{WillCome2017} and Theorem \ref{dream2017}. 
Such a Hodge locus is not $3$-smooth except  for $(\cf,\check\cf)=(1,\pm 1)$  for which 
we have even $4$-smoothness in the case $(8,3,2)$.
The coefficients of the Taylor series in Theorem \ref{InLabelNadasht?} seem to be defined in a reasonable ring,
for instance, for $(n,d)=(4,3),(6,3)$ and some sample truncated Taylor series,  the ring of coefficients is 
$\Z[\frac{1}{d},\zeta_{2d}]$. If so, one may consider them modulo  prime ideals, and in this way, study many related conjectures. 
The tools introduced in this article can be used in order to answer the following question which produces an explicit  
counterexample  to a conjecture of J. Harris:
determine the integer $d$ (conjecturally less than $10$) 
such that  the Noether-Lefschetz
locus of surfaces of degree $d$  (resp degree $<d$) has infinite (resp. finite) number of
special components crossing the Fermat point. Notice that Voisin's counterexample  in
\cite{voisin1991} is for a very big $d$. 
This problem will be studied in subsequent articles. For this  and its 
generalization to higher dimensions one needs to classify linear combination
of linear cycles in the Fermat variety which are semi-irreducible. 
The combinatorics of  arrangement of linear cycles seems to play some role in this question. 
The author's favorite examples in this article have been cubic varieties, see Manin's book \cite{Manin86} for an overview of some
results and techniques. Cubic surfaces carry the famous $27$-lines which is exactly the number \eqref{28.06.2017} of
linear cycles for the Fermat cubic surface.  
Hodge conjecture is known for cubic fourfolds (see \cite{zu77}), and for a restricted class of cubic $8$-folds 
the Hodge conjecture is also known (see \cite{Terasoma90}). In general the Hodge conjecture remains open for cubic hypersurfaces of dimension
$n\geq 6$. Conjecture \ref{CanISayFinally?} makes sense starting from cubic tenfolds whose moduli is $220$-dimensional. It might be useful
to review all the results in this case and to see what one can say more about the algebraic cycle $Z$ in this conjecture.

\def\cprime{$'$} \def\cprime{$'$} \def\cprime{$'$} \def\cprime{$'$}


\begin{thebibliography}{CGGH83}

\bibitem[Aok87]{aoki1987}
Noboru Aoki.
\newblock Some new algebraic cycles on {F}ermat varieties.
\newblock {\em J. Math. Soc. Japan}, 39(3):385--396, 1987.

\bibitem[AS83]{AokiShioda1983}
Noboru {Aoki} and Tetsuji {Shioda}.
\newblock {Generators of the N\'eron-Severi group of a Fermat surface.}
\newblock {Arithmetic and geometry, Pap. dedic. I. R. Shafarevich, Vol. I:
  Arithmetic, Prog. Math. 35, 1-12 (1983).}, 1983.

\bibitem[Blo72]{Bloch1972}
Spencer Bloch.
\newblock Semi-regularity and de{R}ham cohomology.
\newblock {\em Invent. Math.}, 17:51--66, 1972.

\bibitem[CDK95]{cadeka}
Eduardo~H. Cattani, Pierre Deligne, and Aroldo~G. Kaplan.
\newblock On the locus of {H}odge classes.
\newblock {\em J. Amer. Math. Soc.}, 8(2):483--506, 1995.

\bibitem[CG80]{CarlsonGriffiths1980}
James~A. {Carlson} and Phillip~A. {Griffiths}.
\newblock {Infinitesimal variations of Hodge structure and the global Torelli
  problem.}
\newblock {Journees de geometrie algebrique, Angers/France 1979, 51-76
  (1980).}, 1980.

\bibitem[CGGH83]{CGGH1983}
James Carlson, Mark Green, Phillip Griffiths, and Joe Harris.
\newblock Infinitesimal variations of {H}odge structure. {I, II,III}.
\newblock {\em Compositio Math.}, 50(2-3):109--205, 1983.

\bibitem[CHM88]{CHM88}
Ciro Ciliberto, Joe Harris, and Rick Miranda.
\newblock General components of the {N}oether-{L}efschetz locus and their
  density in the space of all surfaces.
\newblock {\em Math. Ann.}, 282(4):667--680, 1988.

\bibitem[{Del}06]{Deligne-HodgeConjecture}
Pierre {Deligne}.
\newblock {The Hodge conjecture.}
\newblock In {\em {The millennium prize problems}}, pages 45--53. Providence,
  RI: American Mathematical Society (AMS); Cambridge, MA: Clay Mathematics
  Institute, 2006.

\bibitem[DMOS82]{dmos}
Pierre Deligne, James~S. Milne, Arthur Ogus, and Kuang-yen Shih.
\newblock {\em Hodge cycles, motives, and {S}himura varieties}, volume 900 of
  {\em Lecture Notes in Mathematics}.
\newblock Springer-Verlag, Berlin, 1982.
\newblock Philosophical Studies Series in Philosophy, 20.

\bibitem[EH87]{EisenbudHarris1985}
David Eisenbud and Joe Harris.
\newblock On varieties of minimal degree (a centennial account).
\newblock In {\em Algebraic geometry, {B}owdoin, 1985 ({B}runswick, {M}aine,
  1985)}, volume~46 of {\em Proc. Sympos. Pure Math.}, pages 3--13. Amer. Math.
  Soc., Providence, RI, 1987.

\bibitem[GPS01]{GPS01}
G.-M. Greuel, G.~Pfister, and H.~Sch\"onemann.
\newblock {\sc Singular} 2.0.
\newblock {A Computer Algebra System for Polynomial Computations}, Centre for
  Computer Algebra, University of Kaiserslautern, 2001.
\newblock {\tt http://www.singular.uni-kl.de}.

\bibitem[Gre88]{green1988}
Mark~L. Green.
\newblock A new proof of the explicit {N}oether-{L}efschetz theorem.
\newblock {\em J. Differential Geom.}, 27(1):155--159, 1988.

\bibitem[Gre89]{green1989}
Mark~L. Green.
\newblock Components of maximal dimension in the {N}oether-{L}efschetz locus.
\newblock {\em J. Differential Geom.}, 29(2):295--302, 1989.

\bibitem[Gri69]{gr69}
Phillip~A. Griffiths.
\newblock On the periods of certain rational integrals. {I}, {II}.
\newblock {\em Ann. of Math. (2) 90 (1969), 460-495; ibid. (2)}, 90:496--541,
  1969.

\bibitem[Gro66]{gro66}
Alexander Grothendieck.
\newblock On the de {R}ham cohomology of algebraic varieties.
\newblock {\em Inst. Hautes \'Etudes Sci. Publ. Math.}, (29):95--103, 1966.

\bibitem[Mac05]{mclean2005}
Catriona Maclean.
\newblock A second-order invariant of the {N}oether-{L}efschetz locus and two
  applications.
\newblock {\em Asian J. Math.}, 9(3):373--399, 2005.

\bibitem[Man86]{Manin86}
Yu.~I. Manin.
\newblock {\em Cubic forms}, volume~4 of {\em North-Holland Mathematical
  Library}.
\newblock North-Holland Publishing Co., Amsterdam, second edition, 1986.
\newblock Algebra, geometry, arithmetic, Translated from the Russian by M.
  Hazewinkel.

\bibitem[Mov11]{ho06-1}
Hossein Movasati.
\newblock {\em Multiple Integrals and Modular Differential Equations}.
\newblock 28th Brazilian Mathematics Colloquium. Instituto de Matem\'atica Pura
  e Aplicada, IMPA, 2011.

\bibitem[Mov17a]{ho13}
Hossein Movasati.
\newblock {\em A {C}ourse in {H}odge {T}heory: with {E}mphasis on {M}ultiple
  {I}ntegrals}.
\newblock Available at
  \href{http://w3.impa.br/~hossein/myarticles/hodgetheory.pdf }{author's
  webpage}. 2017.

\bibitem[Mov17b]{GMCD-NL}
Hossein Movasati.
\newblock Gauss-{M}anin connection in disguise: {N}oether-{L}efschetz and
  {H}odge loci.
\newblock {\em Asian Journal of Mathematics}, 21(3):463--482, 2017.

\bibitem[MV17]{Roberto}
Hossein Movasati and Roberto Villaflor.
\newblock Periods of linear cycles.
\newblock {\em arXiv:1705.00084}, 2017.

\bibitem[Ran81]{Ran1980}
Ziv Ran.
\newblock Cycles on {F}ermat hypersurfaces.
\newblock {\em Compositio Math.}, 42(1):121--142, 1980/81.

\bibitem[Shi79a]{sh79}
Tetsuji Shioda.
\newblock The {H}odge conjecture for {F}ermat varieties.
\newblock {\em Math. Ann.}, 245(2):175--184, 1979.

\bibitem[{Shi}79b]{Shioda1979}
Tetsuji {Shioda}.
\newblock {The Hodge conjecture and the Tate conjecture for Fermat varieties.}
\newblock {\em {Proc. Japan Acad., Ser. A}}, 55:111--114, 1979.

\bibitem[{Shi}81]{Shioda1981}
Tetsuji {Shioda}.
\newblock {On the Picard number of a Fermat surface.}
\newblock {\em {J. Fac. Sci., Univ. Tokyo, Sect. I A}}, 28:725--734, 1981.

\bibitem[Ter90]{Terasoma90}
Tomohide Terasoma.
\newblock Hodge conjecture for cubic {$8$}-folds.
\newblock {\em Math. Ann.}, 288(1):9--19, 1990.

\bibitem[Vil18]{RobertoThesis}
Roberto Villaflor.
\newblock Periods of algebraic cycles.
\newblock {\em Ph.D. thesis}, 2018.

\bibitem[Voi88]{voisin1988}
Claire Voisin.
\newblock Une pr\'ecision concernant le th\'eor\`eme de {N}oether.
\newblock {\em Math. Ann.}, 280(4):605--611, 1988.

\bibitem[Voi89]{voisin89}
Claire Voisin.
\newblock Composantes de petite codimension du lieu de {N}oether-{L}efschetz.
\newblock {\em Comment. Math. Helv.}, 64(4):515--526, 1989.

\bibitem[Voi90]{voisin90}
Claire Voisin.
\newblock Sur le lieu de {N}oether-{L}efschetz en degr\'es {$6$} et {$7$}.
\newblock {\em Compositio Math.}, 75(1):47--68, 1990.

\bibitem[Voi91]{voisin1991}
Claire Voisin.
\newblock Contrexemple \`a une conjecture de {J}. {H}arris.
\newblock {\em C. R. Acad. Sci. Paris S\'er. I Math.}, 313(10):685--687, 1991.

\bibitem[Voi03]{vo03}
Claire Voisin.
\newblock {\em Hodge theory and complex algebraic geometry. {II}}, volume~77 of
  {\em Cambridge Studies in Advanced Mathematics}.
\newblock Cambridge University Press, Cambridge, 2003.
\newblock Translated from the French by Leila Schneps.

\bibitem[Voi13]{voisinHL}
Claire Voisin.
\newblock Hodge loci.
\newblock In {\em Handbook of moduli. {V}ol. {III}}, volume~26 of {\em Adv.
  Lect. Math. (ALM)}, pages 507--546. Int. Press, Somerville, MA, 2013.

\bibitem[Zuc77]{zu77}
Steven Zucker.
\newblock The {H}odge conjecture for cubic fourfolds.
\newblock {\em Compositio Math.}, 34(2):199--209, 1977.

\end{thebibliography}

\end{document}